\documentclass[10pt]{amsart}
\usepackage{lipsum}
\usepackage{mwe}

\usepackage{amsmath,amssymb,amsthm,amscd,amsfonts}
\usepackage{fullpage}
\usepackage{graphicx}
\usepackage{stmaryrd}
\usepackage{float}
\usepackage{hyperref}

\usepackage{algorithm}							
\usepackage[noend]{algpseudocode}					
\usepackage{caption}
\usepackage[margin=1in]{geometry}
\usepackage{url}
\usepackage{xcolor}

\usepackage{mathtools}

\usepackage{amssymb,latexsym}
\usepackage{amsmath,amsfonts,amsthm,amscd,amsxtra,enumerate}
\usepackage{mathtools}
\usepackage{caption}
\usepackage[all]{xy}

\usepackage[margin=1in]{geometry}
\usepackage{epsfig}
\usepackage{pgf,tikz}
\usepackage{multicol}
\usepackage[utf8]{inputenc}
\usepackage{subcaption}

\usepackage{graphicx} 

\usepackage{float} 

\usepackage{tikz}

\usepackage{algpseudocode}
\usepackage{algorithm}
\usepackage{dsfont}

\DeclareMathOperator{\ALPH}{Alph}

\newtheorem{conjecture}{ Conjecture}[section]
\newtheorem{theorem}[conjecture]{ Theorem}
\newtheorem{lemma}[conjecture]{ Lemma}
\newtheorem{observation}[conjecture]{ Observation}
\newtheorem{corollary}[conjecture]{ Corollary}
\newtheorem{proposition}[conjecture]{ Proposition}
\theoremstyle{definition}
\newtheorem{remark}[conjecture]{ Remark}

\newtheorem{question}{ Question}

\newtheorem{definition}[conjecture]{ Definition}

\newtheorem{example}[conjecture]{ Example}

\providecommand\max{\text{\rm max}}

\allowdisplaybreaks

\begin{document}

\title{Bounds on the closed-rich constant of infinite words}

\author{Anuran Maity $^{1, 2}$, Svetlana Puzynina $^{2}$
}

\address
	{$^{1,2}$Department of Mathematics\\ 
	 Indian Institute of Technology Guwahati, 
	  Guwahati, India}
      \email{anuran.maity@gmail.com}

\address
	{$^2$Saint Petersburg State University, 7/9 Universitetskaya nab., 199034 St. Petersburg, Russia}
    \email{s.puzynina@gmail.com}


\begin{abstract} 

A finite word $w$ is called \textit{closed} if it has length at most 1 or it contains a proper factor that occurs both as a prefix and as a suffix but does not have internal occurrences in $w$. An infinite word $u$ is called \textit{closed-rich} if 
 the infimum of all possible ratios between the number of closed factors within any factor $w$ of $u$ and square of the length of $w$ exists and is positive. We define this infimum as the closed-rich constant $C_u$ of the infinite closed-rich word $u$.  Puzynina and Parshina (2024) proved that infinite closed-rich words exist. 
 In this paper, we study possible values of closed-rich constants of infinite closed-rich words. In particular, we estimate the supremum $C_{sup}$ of the closed-rich constants of infinite closed-rich words: we show that 
$C_{sup} \leq 0.165952$. Besides that, we study the closed-rich constant $C_f$ of the  Fibonacci word $f$ and show that $ 0.09519 \leq C_f\leq 0.10893 $. In particular, this gives a lower bound for $C_{sup}$: $ 0.09519 \leq C_{sup}$. 
\end{abstract}
\maketitle

\section{Introduction}

We investigate a problem related to the distribution of closed factors in infinite closed-rich words. A finite word $w$ is called \textit{closed} if it has length at most $1$ or it contains a proper factor that occurs both as a prefix and as a suffix but does not have internal occurrences in $w$. The notion of a closed word is closely related to the notion of a return word. 
The concept of a first return to a factor, often called ``return word", can be seen as a discrete counterpart to the first return map in dynamical systems. It serves as a powerful tool for analyzing symbolic dynamical systems and related fields (see \cite{durand1999substitutional}, \cite{vuillon2001characterization}).
A lot of study has been done on closed words from various perspectives, such as the occurrences of closed factors in prefixes of Sturmian words \cite{de2013open}, the relationship between closed factors and palindromic factors in a word \cite{badkobeh2015number}, closed factorizations of a word \cite{badkobeh2016closed}, reconstruction of a string from its longest closed factor array \cite{bannai2015efficient}, expressing the property of being closed  in first-order logic for automatic sequences \cite{schaeffer2016closed}, closed complexities for the family of Arnoux–Rauzy words \cite{parshina2019open}, a refinement of the Morse–Hedlund theorem  using closed complexities \cite{parshina2020open}, etc.  
Possible numbers of closed factors in a finite word have been studied in \cite{badkobeh2015number}. Most recently, Parshina and Puzynina \cite{parshina2024finite} provided the asymptotics for the maximum number of distinct closed factors a finite word can contain, and introduced the concept of closed-rich words for infinite words.

The study of words rich in specific types of factors is an important topic in modern combinatorics on words. Glen et al. \cite{glen2009palindromic} investigated finite and infinite words rich in palindromes. Bannai et al.  \cite{bannai2017runs} resolved the renowned `runs conjecture', which concerns finite words rich in maximal repetitions. Fici et al.  \cite{fici2017abelian} explored finite and infinite words rich in abelian squares. Recently, Brlek and Li 
\cite{DBLP:journals/combtheory/BrlekL25} solved a long standing conjecture which focuses finite words rich in squares.
Similarly to the concept of these rich words, Parshina and Puzynina \cite{parshina2024finite} defined a word as closed-rich if it is rich in distinct closed factors, i.e., it contains the maximal number of distinct closed factors among words of the same length. They showed that a finite closed-rich word of length $n$ contains asymptotically $\sim \frac{n^2}{6}$ distinct closed factors. An infinite word $u$ is called closed-rich if the constant  $ C_u = \inf \left\{ \frac{ \mathtt{Cl}(w)}{|w|^2} ~:~ w \in  Fac(u)  \right\}$ is positive.
Here $C_u$ is called the closed-rich constant of $u$.
In  \cite{parshina2024finite}, the authors
asked a question about the supremum 
$C_{sup}$ of the closed-rich constants of infinite closed-rich words. Since a finite closed-rich word contains $\sim \frac{n^2}{6}$ distinct closed factors, one has $C_{sup} \leq \frac{1}{6}$. 

In this paper, we show that $C_{sup} \leq 0.165952 $. We remark that we improve the bound obtained in the conference version of this paper \cite{DBLP:conf/cwords/MaityP25}. Besides that, we found a lower bound for  $C_{sup}$ using the Fibonacci word, which has been shown to be closed-rich in \cite{parshina2024finite}. In this paper, we prove that for the Fibonacci word $f$ we have $ \frac{\phi^3+3}{\phi^9} < C_f  \leq  \frac{5\phi+3}{\phi^2(\phi^3+2)^2}$, where $\phi = \frac{1+\sqrt{5}}{2}$ is the golden ratio. The upper bound for $C_f$ is also improved compared to \cite{DBLP:conf/cwords/MaityP25}; moreover, we conjecture that our new  upper bound is actually $C_f$. Thus, we have the following upper and lower bounds for   $C_{sup}$:
  $0.09519 <  C_{sup} \leq 0.165952 $.

The infinite Fibonacci word 
is the most famous Sturmian word providing essential insights to the properties of all Sturmian words. 
Parshina and Puzynina \cite{parshina2024finite} conjectured that finite closed-rich words are cubes or words of exponent close to 3.
Since every sufficiently long finite factor of $f$ always contains sufficiently long factors with an exponent close to $3$
and $f$ often serves as an extremal case and an example in combinatorics on words, 
$f$ is one of the best possible candidates for estimating $C_{sup}$.
 In addition, the techniques and observations developed for the Fibonacci word could be useful for studying other Sturmian words by utilizing the standard sequences associated with them.

The paper is organized as follows. In the next section, we give necessary definitions and notation. In Section \ref{sec3}, we prove the upper bound for $C_{sup}$. In Section \ref{sec4}, we discuss the Fibonacci word. In particular, we provide an explicit formula for the number of closed words in its prefixes, and find  lower and upper bounds for its closed-rich constant mentioned above.
In Section \ref{sec7}, we discuss the bounds of $C_{sup}$, propose some conjectures and open questions.

\section{Preliminaries}\label{sec2}

In this section, we introduce necessary definitions and notation. 
An \textit{alphabet} $\Sigma$ is a finite set of letters. A finite (resp., infinite) \textit{word} $w=w_1w_2w_3 \cdots$ on $\Sigma$ is a finite (resp., infinite) sequence of symbols from $\Sigma$. We let $\ALPH(w)$ denote the set of letters of $w$.
For a finite word $w=w_1 w_2 \cdots w_n$, its \textit{length} is $|w|=n$. We let $\lambda$ denote the \textit{empty word}, and we set by convention $|\lambda|=0$. The set of all finite words over $\Sigma$ is denoted by $\Sigma^*$; the set of all non-empty words over $\Sigma$ is denoted by $\Sigma^+=\Sigma^*\setminus \{\lambda\}$, and the set of all words of length $n$ over $\Sigma$ is denoted by $\Sigma^n$. 
The \textit{reversal} of $w=w_1 w_2 \cdots w_n$ is the word $w^R=w_n w_{n-1} \cdots w_2 w_1$. If $w=w^R$, then $w$ is called a \textit{palindrome}. A word $y$ is a \textit{factor} of a word $u$  if $y = u_i u_{i+1} u_{i+2} \cdots u_{j-1}$ for some positive integers $i, j$ with $i \leq j$. The number $i$ is called an \textit{occurrence} of the factor $y$ in $u$.
In particular, if $i=j$, then the factor $y$ is the empty word $\lambda$ and any index $i$ is its occurrence. The set of all factors of $u$ is denoted by $Fac(u)$. If $w=xux'$ for some $x, u, x' \in \Sigma^*$ and $x=\lambda$ (resp. $x'=\lambda$), then $u$ is a \textit{prefix} (resp. \textit{suffix}) of $w$. The set of all prefixes (resp. suffixes) of $w$ is denoted by $Pref(w)$ (resp. $Suff(w)$). A prefix (resp. suffix) $u$ of $w$ is called \textit{proper} if $u \neq w$ and $u\neq \lambda$. The prefix (resp. suffix) of $w$ of length $i\leq |w|$ is denoted by $Pref_i(w)$ (resp. $Suff_i(w)$).

If a finite word $w$ has a proper prefix $u$ which is also its
suffix, then $u$ is called a \textit{border} of $w$ and $w$ is called a \textit{bordered word}. 
For a finite word $w$ of length $n$, we denote its prefix of length $n-1$ by $w^-$ and its suffix of length $n-1$ by ${}^-w$.
We say that a word $u$ occurs \textit{internally} in $w$ if there exist non-empty words $x$ and $y$ such that $w=xuy$.
A word $u$ is a \textit{conjugate} of $v$ if there exist $x, y \in \Sigma^*$ such that  $v=xy$ and $u=yx$. The set of all conjugates of $u$ is denoted by $C(u)$.
{
If there exists an integer $k$ such that for each $i$ ($i< |w| - k$ if $w$ is finite) the equality $w_i = w_{k+i}$ holds, then $k$ is called a \textit{period} of $w$. 
If $w$ has period $k$ and $k$ is minimal possible, then we say that $w$ has \textit{exponent} $l=\frac{|w|}{k}$, and we write $w=u^l$, where $u$ is the prefix of $w$ of length $k$. The notation $w=u^{k+}$ means that $w$ has exponent $s>k$. 
The word $w$ is \textit{primitive} if is not an integer power of another word.

Let $v = v_0 v_1 v_2 \cdots$ be an infinite word.
Then the \textit{critical exponent} $E(v)$ of $v$ is defined as
$$E(v) = \text{sup}\{e \in \mathbb{Q}  : u^e \text{ is a factor of $v$ for a non-empty word $u$} \}.$$
The \textit{asymptotic critical exponent}  $E^*(v)$ of $v$ is
$$E^*(v) = \limsup_{n \rightarrow \infty} \{ e \in \mathbb{Q} ~:~ u^e \text{ is a factor of $v$ for some } u \text{ of length $n$}\}.$$

A finite word is called \textit{closed} if it has length at most $1$ or it contains a proper factor that occurs as a prefix and as a suffix but does not have internal occurrences. In other words, a closed word of length at least 2 has a border which has exactly two occurrences in the word.
For a finite word $w$, we let $\mathtt{Cl}(w)$ denote the number of distinct closed factors of $w$.
The following upper bound for $\mathtt{Cl}(w)$ is introduced in  \cite{parshina2024finite}:
\begin{theorem}\label{pth}\cite{parshina2024finite}
     For a finite word $w$ of length $n$, 
$\mathtt{Cl}(w) \leq \frac{n^2}{6} + \frac{7}{6}n +1$.
\end{theorem}
Besides that, \cite{parshina2024finite} provides examples of words containing $\sim \frac{n^2}{6}$ distinct closed factors, thus proving the asymptotics $\sim \frac{n^2}{6}$ for the maximum number of distinct closed factors in a word. The definition of a closed-rich word can be extended to infinite words as follows.

\begin{definition}\label{eq:C_u}

For an infinite word $u$, we define the following real number $C_u$: 
$$ C_u = \inf \left\{ \frac{ \mathtt{Cl}(w) }{|w|^2} \; \middle| \; w \in  Fac(u), |w|\geq 1 \right\}.$$ 
If $C_u$ is positive, then the word $u$ is called \textit{closed-rich}, and $C_u$ is called the closed-rich constant of $u$.
\end{definition}











\begin{remark} \label{remsup}
In \cite{parshina2024finite}, infinite closed-rich words have been defined in a slightly different although equivalent way. According to  \cite{parshina2024finite}, an infinite word $u$ is closed-rich if there exists a positive constant $C'$ such that for each $n$, each factor of length $n$ of $u$ contains at least $C'n^2$ distinct closed factors. It is not hard to show that the two definitions are equivalent, and in fact the constant $C_u$ from Definition \ref{eq:C_u} can be defined as the supremum of the constants $C'$:
$$ C_u = \sup \left\{ C': \mathtt{Cl}(w)\geq C'|w|^2 \mbox{ for each } w \in  Fac(u)  \right\}.$$

\end{remark}

We now recall a basic fact of combinatorics on words that will be used multiple times in this paper. 

\begin{lemma} \label{4lk}\cite{Schutz62}
Let $u,v,w \in \Sigma^{+}$.
\begin{itemize}
     \item If $uv=vu$, then $u$ and $v$ are powers of a common primitive word. 
     \item If $uv=vw$, then for $k \geq 0$, $x \in \Sigma^{+}$ and $y \in \Sigma^{*}$, $u=xy$, $v=(xy)^{k}x$, $w=yx$.
 \end{itemize}
\end{lemma}

We will use the following notation. We let $\mathtt{Cl}'(w)$ denote the number of closed factors of $w$ of length at least 2 (“long” closed factors), and $\mathtt{Cl}^{0}(w)$ denote the number of closed factors of $w$ of length at most $1$ (“short” closed factors). So,
$\mathtt{Cl}(w) = \mathtt{Cl}'(w) + \mathtt{Cl}^{0}(w)$. If $w$ is closed and $v$ is its longest border, then we write $\mathtt{LB}(w)=v$. 

\section{Upper bound for the closed-rich constant of any infinite closed-rich word }\label{sec3}
In this section, we 
show that  the closed rich constant of any infinite closed-rich word is smaller than or equal to $ 0.165952 $. 
We will use the following fact:

\begin{lemma}\label{lem0}
    Let $w, u \in \Sigma^+$ such that $\ALPH(u)=\ALPH(w)$ and $w=ux$ where $x$ is a letter. If $t$ is the largest repeated suffix of $w$ and $z$ is the largest repeated suffix of $t$, then $\mathtt{Cl}(w)- \mathtt{Cl}(u) = |t| - |z|.$ 
\end{lemma}
\begin{proof}
    Since $\ALPH(u)=\ALPH(w)$, there always exists a non-empty repeated suffix of $w$ and  $\mathtt{Cl}(w)- \mathtt{Cl}(u)= \mathtt{Cl}'(w)- \mathtt{Cl}'(u)$.
    Then, $\mathtt{Cl}'(w)- \mathtt{Cl}'(u) \leq |t| - |z|$. Now, the repetition of $t$ implies that for each suffix $p$ of $t$, there always exists a closed factor $\beta$ of $w$ such that $\mathtt{LB}(\beta)=p$.
    
    Now, we show that $\mathtt{Cl}'(w)- \mathtt{Cl}'(u) = |t| - |z|$. 
    Assume the converse, i.e., that $\mathtt{Cl}'(w)- \mathtt{Cl}'(u) < |t| - |z|$. Then, there exists a long closed factor $\alpha$ of $u$ such that $\alpha$ is a suffix of $w$ and $|z|+1 \leq |\mathtt{LB}(\alpha)| \leq |t|$.
   Since $t$ is the largest repeated suffix of $w$, we have $|\alpha| \leq |t|$. But as $|\mathtt{LB}(\alpha)| \geq |z|+1$ and $\alpha$ is a suffix of $w$, $z$ is not the largest repeated suffix of $t$, which is a contradiction.
    Thus,  there does not exist any long closed factor $\alpha$ of $u$ such that $\alpha$ is a suffix of $w$ and $|z|+1 \leq |\mathtt{LB}(\alpha)| \leq |t|$. Therefore, $\mathtt{Cl}'(w)- \mathtt{Cl}'(u) = |t| - |z|$.
\end{proof}

We will also use a weaker form of Lemma \ref{lem0}:

\begin{corollary}\label{cor0}
   Let $w, u \in \Sigma^+$ such that $\ALPH(u)=\ALPH(w)$ and $w=ux$, where $x$ is a letter. If $t$ is the largest repeated suffix of $w$, then $\mathtt{Cl}(w)- \mathtt{Cl}(u) \leq|t|$. 
\end{corollary}

 The following lemma is symmetric to Lemma \ref{lem0}:

\begin{lemma}\label{lem01}
    Let $w, u \in \Sigma^+$  such that $\ALPH(u)=\ALPH(w)$ and $w=xu$ where $x$ is a letter. If $p$ is the largest repeated prefix of $w$ and $q$ is the largest repeated prefix of $p$, then $\mathtt{Cl}(w)- \mathtt{Cl}(u) = |p| - |q|.$
\end{lemma}

\begin{lemma}\label{lem02}
    Let $w$ and $u$ be two non-empty words such that $\ALPH(u) \neq \ALPH(w)$.
    \begin{itemize}
        \item If $w=ux$ for some $x \in \Sigma$, then $\mathtt{Cl}(w)- \mathtt{Cl}(u) = 1$.
        \item If $w=xu$ for some $x \in \Sigma$, then $\mathtt{Cl}(w)- \mathtt{Cl}(u) = 1$.
    \end{itemize}
\end{lemma}
\begin{proof}
    For both cases, since $\ALPH(u) \neq \ALPH(w)$, $x \notin \ALPH(u)$. This implies $\mathtt{Cl}(w)= \mathtt{Cl}(u) + 1$, where the new factor in $w$ compare to $u$ is the letter $x$.
\end{proof}

The main result of this section is the following theorem giving an upper bound for the closed-rich constant of infinite closed-rich words.

\begin{theorem}\label{mthsd}
The  closed-rich constant of an infinite closed-rich word  is less than or equal to $\frac{967}{5827}$.
\end{theorem}



The main strategy of the proof is as follows. We will distinguish between several cases depending on the asymptotic critical exponent of the infinite word. We will say that an infinite word $v$ \textit{has long $\alpha$-powers} if for any $m \in \mathbb{N}$, there exists a factor $u$ of $v$ of length $n\geq m$ such that the exponent of $u$ is at least $\alpha$. Now, if $\alpha$ is the asymptotic critical exponent of an infinite word $v$, then two cases are possible: either $v$ has long $\alpha$-powers and does not have long $\alpha^+$-powers, or $v$ has long factors with exponents $\geq (\alpha - \epsilon)$ for any $\epsilon >0$ and does not have long $\alpha$-powers. We split possible values of the asymptotic critical exponent of infinite closed-rich words into four intervals and obtain an upper bound for the closed-rich constant for each interval, and the upper bound in Theorem \ref{mthsd} is given by the maximal of the bounds obtained for each interval.
For each of the intervals we treat the two cases above in separate propositions. The interval between $1$ and $1.4$ is considered in Propositions \ref{thdf33} and \ref{thdf331}, the interval  between $1.4$ and $2$ is considered in Propositions \ref{thdf22} and \ref{thdf221}, the interval between $2$ and $2.5$ is considered in Propositions \ref{thdf1} and \ref{qa11}, and finally asymptotic critical exponents greater than $2.5$ are considered in Proposition \ref{cont1} (the intervals are closed or open on the right and left depending on the case). For each of the intervals the proofs in the two cases are symmetric, so we provide only proofs for the first case, with the exception of Proposition  \ref{qa11} which we provide to demonstrate using the proof technique from the symmetric Proposition \ref{thdf1}. Finally, the proof of Theorem \ref{mthsd} is given by combining the above propositions and the fact that infinite words with unbounded critical exponent are not closed-rich \cite{parshina2024finite}. 

\begin{proposition}\label{thdf33}
    Let $v$ be an infinite closed-rich word such that it has long $\alpha$-powers but it does not have long $\alpha^+$-powers, where  $1 < \alpha \leq 1.4$. Then the closed-rich constant is at most $\frac{1}{7}$. 
\end{proposition}


\begin{proof}
Since $v$  has long $\alpha$-powers and does not have long $\alpha^+$-powers, then for each $\epsilon >0$, there exists a positive integer $M$ (depending on $\epsilon$) such that the length of any factor of $v$ with exponent $\alpha + \epsilon$ is at most $M$.
    For a fixed $\delta >0$, we set $N=\lceil \frac{M}{\delta}\rceil$. Let $w$ be a 
    factor of $v$ 
    of length $n\geq N$; by the choice of $n$ we have that for each $\epsilon>0$ each $(\alpha+\epsilon)$-power factor of $w$ has length at most $n \delta$.
      For an integer $i$, we let $x^{(i)}$ denote the prefix of $w$ of length $i$ and $t^{(i)}$ denote the largest repeated suffix of $x^{(i)}$.
    If $t^{(i)} = \lambda$, then $\mathtt{Cl}(x^{(i)})- \mathtt{Cl}(x^{(i-1)}) = 1 \leq n \delta$. 
    Now, consider $t^{(i)} \in \Sigma^+$.
    If two rightmost occurrences of
    $t^{(i)} $ in $x^{(i)} $ intersect, then $t^{(i)} =yp=pg$ for some $y, p, g \in \Sigma^+$. Then the exponent of $ypg$ is greater than equal to $2$. This implies that $|ypg| < n\delta$. Hence $|t^{(i)}| < n\delta$.
   Now suppose that two rightmost occurrences of $t^{(i)}$ in $x^{(i)} $ do not intersect. 
    Then, for some $y^{(i)} \in \Sigma^*$, we have that $t^{(i)} y^{(i)}  t^{(i)} $ is a suffix of $x^{(i)} $. 
    Now, $t^{(i)} y^{(i)} t^{(i)} = (t^{(i)} y^{(i)})^{\frac{|t^{(i)}y^{(i)}t^{(i)}|}{|t^{(i)} y^{(i)}|}}$.
    If $\frac{|t^{(i)} y^{(i)} t^{(i)} |}{|t^{(i)} y^{(i)} |} > \alpha$, then $|t^{(i)} y^{(i)} t^{(i)}| < n\delta $, i.e., $|t^{(i)}| < n \delta$.
    If $\frac{|t^{(i)} y^{(i)} t^{(i)}|}{|t^{(i)} y^{(i)}|} \leq \alpha$, then $|t^{(i)}| \leq \frac{\alpha -1}{\alpha}|t^{(i)} y^{(i)} t^{(i)}| \leq \frac{\alpha -1}{\alpha} i$ as $|t^{(i)} y^{(i)} t^{(i)}| \leq i$.   
  
    Combining all cases and using Lemma \ref{lem0}, 
    we have  $$\mathtt{Cl}(x^{(i)})- \mathtt{Cl}(x^{(i-1)}) \leq |t^{(i)}| \leq n\delta + \frac{\alpha -1}{\alpha} i.$$
    Thus, $$\mathtt{Cl}(w) \leq 1+  \sum_{i=1}^{n} \left(n\delta +  \frac{\alpha -1}{\alpha} i\right), $$ (here, $1$ comes from the empty word, which is closed by definition), i.e., $$\mathtt{Cl}(w) \leq 1+   n^2\delta + \frac{(\alpha-1) (n^2+n)}{2\alpha}.$$ 
    
 Then, $$ \frac{ \mathtt{Cl}(w) }{n^2} \leq \frac{1}{n^2} + \delta + \frac{\alpha -1}{2\alpha} \left(1+\frac{1}{n}\right).$$
 
    Thus, $$  \inf \left\{ \frac{ \mathtt{Cl}(w) }{n^2}~:~ w \in \Sigma^n \cap Fac(v) ~\&~ n \geq N \right\} \leq \delta + \frac{\alpha -1}{2\alpha}.$$

    This implies by Definition \ref{eq:C_u} that for the closed-rich constant $C_v$ of $v$, we have $C_v\leq \delta + \frac{\alpha -1}{2\alpha}$. 

    The function $f(\alpha)= \frac{\alpha -1}{2\alpha}$ on the interval $[1,1.4]$ takes its maximum when $\alpha=1.4$.
    Since we can choose $\delta$ to be as small as we want, the closed-rich constant of $v$ is at most $f(1.4)=\frac{1}{7}$.    
   \end{proof}


Similarly to the proof of Proposition \ref{thdf33}, we can prove the following its symmetric version:

\begin{proposition}\label{thdf331}
    Let $\alpha$ be the asymptotic critical exponent of an infinite closed-rich word  $v$ such that for any $\epsilon>0$,  it has long factors with exponent $\geq \alpha-\epsilon $  but it does not have long $\alpha$-powers, where  $1< \alpha \leq 1.4$. Then, the closed-rich constant is at most $\frac{1}{7}$.
\end{proposition}

\begin{proposition}\label{thdf22}
    Let $v$ be an infinite closed-rich word of asymptotic critical exponent $\alpha$, where  $1.4 < \alpha < 2$. Suppose that it has long $\alpha$-powers but it does not have long $\alpha^+$-powers. Then the closed-rich constant of $v$ is at most $ \frac{13}{80}$.
\end{proposition}
\begin{proof}
    Since $v$ has long $\alpha$-powers but it does not have long $\alpha^+$-powers, then for each $\epsilon >0$, there exists a positive integer $M$ 
    such that the length of any factor of $v$ with exponent $\alpha + \epsilon$ is at most $M$.
    For a fixed $\delta >0$, we set $N=\lceil \frac{M}{\delta}\rceil$. Let $w$ be a $\alpha$-power factor of $v$ of length $n\geq N$ (such $w$ exists in $v$ by the condition of the proposition). Note that for each $(\alpha+\epsilon)$-power factor $z$ of $w$, 
    where $\epsilon>0$, its length is at most $M\leq \delta N\leq \delta n $; we will use this fact several times in the proof.
    Then $w$ is of the form $w=upu$, where $u, p \in \Sigma^+$, 
    $|u|=l$, $|p|=r$ and $\frac{2l+r}{l+r}=\alpha$. 
    Since $|w|=n=2l+r$, we have $l=\frac{n(\alpha-1)}{\alpha}$ and $r=\frac{n(2-\alpha)}{\alpha}$.
    Consider some occurrence of $w$ in $v$ and its extension to the right by a factor of length $\xi n $ for some $\xi>0$  (we will choose the value of $\xi$ later). 
    We let $x$ denote this extension, and $w'$ the extended factor, i.e. $w'=upu x \in Fac(v)$, 
    where $|x|=\xi n$. 
    We find an upper bound for the number of distinct closed factors in $w'$, which will give us an upper bound for the closed-rich constant of $v$.
    Let $x^{(i)}$ 
    be a prefix of $x$ of length $i$ and $t^{(i)}$  be the largest repeated suffix of $upu x^{(i)}$.
     If $t^{(i)} = \lambda$, then $\mathtt{Cl}(upux^{(i)}) - \mathtt{Cl}(upux^{(i)-}) = 1$.
     Now, consider $t^{(i)} \in \Sigma^+$. 
    Let us denote the rightmost $t^{(i)}$ in $upu x^{(i)}$ by $t^{(i1)}$ and the rightmost $t^{(i)}$ in $upu x^{(i)-}$ by $t^{(i2)}$. 
    We consider two cases depending on whether $t^{(i1)}$ and $t^{(i2)}$ intersect or not, and in each case we get an upper bound on the length of $t^{(i)}$. We will later use these bounds to obtain an upper bound for the number of distinct closed-rich factors of $w'$, which immediately gives an upper bound for the closed-rich constant of $v$. 
     
    If $t^{(i1)}$ and $t^{(i2)}$ intersect (see Fig. \ref{wes1}), 
    then $t^{(i2)}=s^{(1)} s^{(2)}$ and $t^{(i1)}=s^{(2)} s^{(3)}$ for some $s^{(1)}, s^{(2)}, s^{(3)} \in \Sigma^+$. Then, the exponent of $s^{(1)} s^{(2)} s^{(3)}$ is greater than or equal to $2$. Since all factors of $v$ of exponent greater or equal to $2$ are of length at most $n\delta$, we have $|s^{(1)} s^{(2)} s^{(3)}| < n\delta$. Since $|t^{(i)}|<|s^{(1)} s^{(2)} s^{(3)}|$, we also have $|t^{(i)}| < n\delta$.


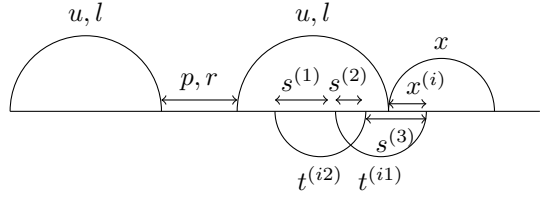
\begin{figure}[h]
    \centering
        \begin{tikzpicture}
        \draw (1,0) arc[start angle=0, end angle=180, radius=1] node[midway, above] {$u, l$};
        \draw (4,0) arc[start angle=0, end angle=180, radius=1] node[midway, above] {$u, l$};
        \draw (5.4,0) arc[start angle=0, end angle=180, radius=.7] node[midway, above] {$x$};
        \draw (4.5,0) arc[start angle=0, end angle=-180, radius=.6] node[midway, below] {$t^{(i1)}$};
         \draw (3.7,0) arc[start angle=0, end angle=-180, radius=.6] node[midway, below] {$t^{(i2)}$};
        \draw[<->] (2.5,0.15) -- (3.2,0.15) node[midway, above] {$s^{(1)}$};
        \draw[<->] (3.3,0.15) -- (3.65,0.15) node[midway, above] {$s^{(2)}$};
        \draw[<->] (3.7,-0.1) -- (4.5,-0.1) node[midway, below] {$s^{(3)}$};
        \draw[<->] (4,0.1) -- (4.5,0.1) node[right, above] {$x^{(i)}$};
        \draw[<->] (1,0.15) -- (2,0.15) node[midway, above] {$p, r$};
        \draw[-] (-1,0) -- (6,0);
        \end{tikzpicture}
        \caption{Illustration to the proof of Proposition \ref{thdf22} in the case when $t^{(i1)}$ and $t^{(i2)}$ intersect.}
        \label{wes1}
\end{figure}

\begin{figure}[h]   
        \centering
        \begin{tikzpicture}
        \draw (1,0) arc[start angle=0, end angle=180, radius=1] node[midway, above] {$u, l$};
        \draw (4,0) arc[start angle=0, end angle=180, radius=1] node[midway, above] {$u, l$};
        \draw (5.4,0) arc[start angle=0, end angle=180, radius=.7] node[midway, above] {$x$};
        \draw (4.5,0) arc[start angle=0, end angle=-180, radius=.6] node[midway, below] {$t^{(i1)}$};
         \draw (3,0) arc[start angle=0, end angle=-180, radius=.6] node[midway, below] {$t^{(i2)}$};
         \draw[<->] (1.8,0.15) -- (2.48,0.15) node[midway, above] {$g$};
         \draw[-] (2.48,0.1) -- (2.48,-0.1) node[midway, above] {};
        \draw[<->] (3.3,0.15) -- (3.98,0.15) node[midway, above] {$g$};
        \draw (3.3,0) arc[start angle=0, end angle=180, radius=.4] node[midway, above] {$y$};
        \draw[<->] (4,0.1) -- (4.5,0.1) node[right, above] {$x^{(i)}$};
        \draw[<->] (1,.3) -- (2,0.3) node[midway, above] {$p, r$};
        \draw[-] (-1,0) -- (6,0);
        \end{tikzpicture}
        \caption{Illustration to the proof of Proposition \ref{thdf22}, Case 1.1.}
        \label{wes2}
\end{figure}

    Now suppose that $t^{(i1)}$ and $t^{(i2)}$  do not intersect. Then as $t^{(i1)}, t^{(i2)} \in Fac(upux^{(i)})$ and $|upux^{(i)}|=2l+r+i$, we have  $|t^{(i)}| \leq l + \frac{i+r}{2} $.   
    

    Now, we consider the following cases, depending on the length of $t^{(i)}$: 
   
   \textbf{Case 1}: $|t^{(i)}| \leq l+i$. 
        
        In this case we have either  $|t^{(i)}| \leq i$ or $t^{(i1)}=gx^{(i)}$ for some $g \in \Sigma^+$.
         In the first case we will use this bound for $|t^{(i)}|$, so we now consider the case $g \in \Sigma^+$. We consider two subcases depending on the occurrence of $t^{(i2)}$:
         
         
            \textbf{Case 1.1}: $t^{(i2)} \in Fac(pu)$ (see Fig. \ref{wes2}).
            
            In this case, since $g \in Pref(t^{(i2)})$, we have $gyg \in Suff(pu)$ for some $y \in \Sigma^+$. Now, $gyg=(gy)^{\frac{|g|+|y|+|g|}{|g|+|y|}}$.
        If $\frac{|g|+|y|+|g|}{|g|+|y|}>\alpha$, then $|gyg|<n \delta$, so we also have $|g|< n \delta$. Then $|t^{(i)}| < n\delta +i$.
        
        If $\frac{|g|+|y|+|g|}{|g|+|y|} \leq \alpha$, then $\frac{(2-\alpha)|g|}{\alpha-1} \leq |y|$. Since $gyg \in Fac(pu)$, $|g|+|y|+|g| \leq l+r$. Then, as $\frac{(2-\alpha)|g|}{\alpha-1} \leq |y|$, we have 
        $|g|+ \frac{(2-\alpha)|g|}{\alpha-1} +|g| \leq |g|+|y|+|g| \leq l+r$. Since $l+r=\frac{n}{\alpha}$, we get $|g| \leq \frac{n (\alpha-1)}{\alpha^2}$. Thus, $ |t^{(i)}| = |g|+i  \leq \frac{n (\alpha-1)}{\alpha^2} + i$.

\textbf{Case 1.2}: $t^{(i2)} \in Fac(upu)$ and $t^{(i2)} \notin Fac(pu)$ (see Fig. \ref{wes3}).
            
            In this case, we have either $t^{(i2)}$ is completely contained in the first occurrence of $u$ in $w'$ (which is not possible as there would be another occurrence of $t^{(i)}$ inside the rightmost occurrence of $u$, and this occurrence of $t^{(i)}$ would be the second right occurrence of $t^{(i)}$ in $w'$ instead of $t^{(i2)}$), or $t^{(i2)} = \beta^{(1)} \beta^{(2)}$ for some non-empty words $\beta^{(1)}$ and $\beta^{(2)}$ such that $\beta^{(1)} \in Suff(u)$ and $\beta^{(2)} \in Pref(pu)$.
            Since $w$ is an $\alpha$-power and not an $\alpha^+$-power, we have that $w_{l+1}\neq w_{2l+r+1}$. 
            So, we have $|g| \neq |\beta^{(1)}|$. 

\textbf{Case 1.2.1}: If $|\beta^{(1)}| < 2|g|$, then the occurrence of $g$ as a prefix of $t^{(i2)}$ and the occurrence of $g$ as a suffix of the first occurrence of $u$ in $w$ intersect. Thus we have a factor of exponent greater than $2$, so $|g|<n\delta$. This implies $|t^{(i)}| < n\delta + i$.

                \textbf{Case 1.2.2}:  If $|\beta^{(1)}| \geq 2|g|$, then as $g\in Pref(t^{(i2)})$ and $g \in Suff(u)$, $gqg$ is a suffix of $u$ for some $q \in \Sigma^*$. This implies $gqg \in Pref(t^{(i2)})$, i.e., $qg \in Pref(x^{(i)})$. Then, $gqgqg \in Fac(ux^{(i)})$. Since exponent of $gqgqg$ is greater than or equal to $2$, $|gqgqg|<n \delta$, i.e., $|g|< n \delta$. Thus, $|t^{(i)}| < n\delta + i$.
        \begin{figure}[h]
        \centering
        \begin{tikzpicture}
        \draw (1,0) arc[start angle=0, end angle=180, radius=1] node[midway, above] {$u, l$};
        \draw (4,0) arc[start angle=0, end angle=180, radius=1] node[midway, above] {$u, l$};
        \draw (5.4,0) arc[start angle=0, end angle=180, radius=.7] node[midway, above] {$x$};
        \draw (5.1,0) arc[start angle=0, end angle=-180, radius=1] node[midway, below] {$t^{(i1)}$};
         \draw (1.9,0) arc[start angle=0, end angle=-180, radius=1] node[midway, below] {$t^{(i2)}$};
         \draw[] (.4,0.01) -- (.5,0.01) node[midway, above] {$\beta^{(1)}$};
        \draw[] (3.5,0.01) -- (3.49,0.01) node[midway, above] {$g$};
        \draw[] (4.5,0.01) -- (4.49,0.01) node[midway, above] {$x^{(i)}$};
        \draw[<->] (1,.3) -- (2,0.3) node[midway, above] {$p, r$};
        \draw[-] (-1,0) -- (6,0);
        \end{tikzpicture}
        \caption{Illustration to the proof of Proposition \ref{thdf22}, Case 1.2.} 
        \label{wes3}
        \end{figure}



\textbf{Case 2}: $l + i < |t^{(i)}| \leq l+\frac{i+r}{2}$.
        
        In this case, $t^{(i1)} = y'ux^{(i)}$ for some $y' \in \Sigma^+$.
 
 \textbf{Case 2.1}: $t^{(i2)} \in Fac(p)$ (see Fig. \ref{wes41}).
             
             In this subcase, as $u \in Fac(t^{(i2)})$, $r>l$. Then $1.4 < \alpha <1.5$. This implies $l<r<\frac{3l}{2}$. Now, $u \in Fac(t^{(i2)})$ and $u \in Suff(w)$ implies $uu'u \in Suff(w)$ for some $u' \in \Sigma^+$. Since $r<\frac{3l}{2}$ and $|u|=l$, we have  $|u'|<l$. Then the  exponent of $uu'u$ is $>1.5$, i.e., the exponent of $uu'u$ is greater than $\alpha$. This implies $|uu'u|<n \delta$. Since $|u'|> |y'|$, we have $|t^{(i)}|=|y'|+|u|+i \leq |u'|+|u|+i \leq n \delta + i$.


\begin{figure}[h]
    \centering
        \begin{tikzpicture}
        \draw (-0.8,0) arc[start angle=0, end angle=180, radius=1] node[midway, above] {$u, l$};
        \draw (4,0) arc[start angle=0, end angle=180, radius=1] node[midway, above] {$u, l$};
        \draw (5.4,0) arc[start angle=0, end angle=180, radius=.7] node[midway, above] {$x$};
        \draw (4.2,0) arc[start angle=0, end angle=-180, radius=1.2] node[midway, below] {$t^{(i1)}$};
         \draw (1.6,0) arc[start angle=0, end angle=-180, radius=1.1] node[midway, below] {$t^{(i2)}$};
         \draw[-] (1.3,0.1) -- (1.3,-0.1);
         \draw[<->] (1.35,0.2) -- (1.6,0.2) node[right, above] {$x^{(i)}$};
         \draw[<->] (-0.3,0.15) -- (1.3,0.15) node[midway, above] {$u$};
         \draw[-] (-0.3,0.1) -- (-0.3,-0.1);
          \draw[-] (-0.4,0.05) -- (-0.4,0.05) node[midway, above] {$y'$};
        \draw[-] (1.9,0.01) -- (1.9,0.01) node[midway, above] {$y'$};
        \draw[] (4.2,0.01) -- (4.3,0.01) node[right, above] {$x^{(i)}$};
        \draw (2,0) arc[start angle=0, end angle=180, radius=1.4] node[midway, above] {$p, r$};
         \draw (2,0) arc[start angle=0, end angle=-180, radius=.35] node[midway, below] {$u'$};
        \draw[-] (-2.8,0) -- (5.5,0);
        \end{tikzpicture}
        \caption{{Illustration to the proof of Proposition \ref{thdf22}, Case 2.1}.}
        \label{wes41}
    \end{figure}

    \begin{figure}[h]
        \centering
        \begin{tikzpicture}
        \draw (.5,0) arc[start angle=0, end angle=180, radius=1] node[midway, above] {$u, l$};
        \draw (4,0) arc[start angle=0, end angle=180, radius=1] node[midway, above] {$u, l$};
        \draw (5.4,0) arc[start angle=0, end angle=180, radius=.7] node[midway, above] {$x$};
        \draw (4.3,0) arc[start angle=0, end angle=-180, radius=1.3] node[midway, below] {$t^{(i1)}$};
         \draw (1.4,0) arc[start angle=0, end angle=-180, radius=1.3] node[midway, below] {$t^{(i2)}$};
        \draw[-] (1.9,0.01) -- (1.9,0.01) node[midway, above] {$y'$};
        \draw[<->] (1.8,-0.15) -- (2.8,-0.15) node[midway, below] {$g^{(1)}$};
        \draw[<->] (-1.1,0.15) -- (0.46,0.15) node[midway, above] {$g^{(1)}$};
        \draw[<->] (0.6,0.15) -- (1.3,0.15) node[midway, above] {$g^{(2)}$};
        \draw[] (4.2,0.01) -- (4.5,0.01) node[midway, above] {$x^{(i)}$};
         \draw[<->] (2.8,0.15) -- (4,0.15) node[midway, above] {$g^{(1)}$};
        \draw (2,0) arc[start angle=0, end angle=180, radius=.75] node[midway, above] {$p, r$};
        \draw[-] (-1.5,0) -- (5.5,0);
        \end{tikzpicture}
        \caption{{Illustration to the proof of Proposition \ref{thdf22}, Case 2.2. 
        }}
        \label{wes4}
\end{figure}

\textbf{Case 2.2}: $t^{(i2)} \in Fac(up)$ and $t^{(i2)} \notin Fac(p)$ (see Fig. \ref{wes4})
             
             In this subcase, $t^{(i2)} =g^{(1)} g^{(2)}$  for some nonempty $g^{(1)}, g^{(2)}$, such that $g^{(1)} \in Suff(u)$ and $g^{(2)}  \in  Pref(p)$. Then $g^{(1)} \in  Pref(t^{(i1)})$. 
             
             Now, if the occurrence of $g^{(1)}$ as a prefix of $t^{(i1)}$ and the occurrence of $g^{(1)}$ as a suffix of the second occurrence of $u$ intersect, then the exponent of the factor given by this intersection of the two occurrences of $g^{(1)}$ (in other words, the factor $t^{(i1)} (x^{(i)})^{-1}$) is greater than equal to $2$, 
              so its length is at most $ n\delta$, and hence $|t^{(i)}|< n\delta +i$.
             
             If $g^{(1)} \in  Pref(t^{(i1)})$ and $g^{(1)} \in  Suff(u)$ do not intersect, then $t^{(i1)} = g^{(1)} y'' g^{(1)} x^{(i)}$ for some $y'' \in \Sigma^*$. Now, we have the following cases:

\textbf{Case 2.2.1}: $|g^{(1)}|\geq |y'|$ (see Fig. \ref{wes4122}). 
                 
                 Then,  $y'' \in Fac(u)$. So, $y'' g^{(1)} \in Suff(u)$. 
                 Then as $g^{(1)} y'' g^{(1)} \in Pref(t^{(i2)})$,  $y'' g^{(1)} y'' g^{(1)} \in Fac(up)$. Since exponent of $y'' g^{(1)} y'' g^{(1)} $ is $2$, $|y'' g^{(1)} y'' g^{(1)} | < n \delta$. This implies $|t^{(i)}| < n \delta + i$.
                 
                  \textbf{Case 2.2.2}: $|g^{(1)}| < |y'|$ (see Fig. \ref{wes42123}). 
                  
                  Then, as $y' \in Pref(t^{(i2)})$, we have $u \in Fac(p)$.  Therefore, $r>l$, and so similarly to the Case 2.1, we have $1.4 < \alpha <1.5$. This implies $l<r<\frac{3l}{2}$.
                  Now, $u \in Fac(p)$ and $u \in Suff(w)$ implies $uu''u \in Suff(w)$ for some $u'' \in \Sigma^+$. Since $r<\frac{3l}{2}$ and $|u|=l$,  $|u''|<l$. Then the exponent of $uu''u$ is greater than $1.5$, i.e., greater than $\alpha$. This implies $|uu''u|<n \delta$. Since $|u''|> |y'|$, we have $|t^{(i)}|=|y'|+|u|+i \leq |u''|+|u|+i \leq n \delta + i$.
                  
\begin{figure}[h]
    \centering
        \begin{tikzpicture}
        \draw (.5,0) arc[start angle=0, end angle=180, radius=1] node[midway, above] {$u, l$};
        \draw (4,0) arc[start angle=0, end angle=180, radius=1] node[midway, above] {$u, l$};
        \draw (5.4,0) arc[start angle=0, end angle=180, radius=.7] node[midway, above] {$x$};
        \draw (4.3,0) arc[start angle=0, end angle=-180, radius=1.3] node[midway, below] {$t^{(i1)}$};
         \draw (1.4,0) arc[start angle=0, end angle=-180, radius=1.2] node[midway, below] {$t^{(i2)}$};
        \draw[-] (1.9,0.01) -- (1.9,0.01) node[midway, above] {$y'$};
        \draw[<->] (-1,0.15) -- (0.46,0.15) node[midway, above] {$g^{(1)}$};
        \draw[<->] (0.6,0.15) -- (1.3,0.15) node[midway, above] {$g^{(2)}$};
        \draw[] (4.2,0.01) -- (4.4,0.01) node[right, above] {$x^{(i)}$};
        \draw[<->] (1.8,-0.15) -- (2.7,-0.15) node[midway, below] {$g^{(1)}$};
         \draw[<->] (2.9,0.15) -- (4,0.15) node[midway, above] {$g^{(1)}$};
         \draw[-] (2.7,0.15) -- (2.7,-0.15) node[midway, above] {};
         \draw[-] (2.9,0.15) -- (2.9,-0.15) node[midway, above] {};
         \draw[] (2.9,0.15) -- (2.9,0.15) node[midway, above] {$y''$};
        \draw (2,0) arc[start angle=0, end angle=180, radius=.75] node[midway, above] {$p, r$};
        \draw[-] (-1.5,0) -- (5.5,0);
        \end{tikzpicture}
        \caption{{Illustration to the proof of Proposition \ref{thdf22}, Case 2.2.1 }}
        \label{wes4122}
    \end{figure}
    
    \begin{figure}[h]
        \centering
        \begin{tikzpicture}
        \draw (.6,0) arc[start angle=0, end angle=180, radius=1] node[midway, above] {$u, l$};
        \draw (6,0) arc[start angle=0, end angle=180, radius=1] node[midway, above] {$u, l$};
         \draw (7,0) arc[start angle=0, end angle=180, radius=.5] node[midway, above] {$x$};
        \draw (6.3,0) arc[start angle=0, end angle=-180, radius=1.5] node[midway, below] {$t^{(i1)}$};
        \draw (2.7,0) arc[start angle=0, end angle=-180, radius=1.3] node[midway, below] {$t^{(i2)}$};
        \draw[<->] (3.37,-0.15) -- (4,-0.15) node[midway, below] {$y'$};
        \draw[<->] (0.17,-0.15) -- (0.8,-0.15) node[midway, below] {$y'$};
        \draw[<->] (0.82,-0.15) -- (2.5,-0.15) node[midway, below] {$u$};
        \draw[-] (2.5,0.09) -- (2.5,-0.09) node[midway, above] {};
        \draw[-] (3.7,0.09) -- (3.7,-0.09) node[midway, above] {};
        \draw (4,0) arc[start angle=0, end angle=-180, radius=.74] node[midway, below] {$u''$};
        \draw[<->] (3.3,0.15) -- (3.7,0.15) node[midway, above] {$g^{(1)}$};
         \draw[<->] (0,0.15) -- (0.47,0.15) node[midway, above] {$g^{(1)}$};
        \draw (4,0) arc[start angle=0, end angle=180, radius=1.7] node[midway, above] {$p, r$};
        \draw[-] (-1.4,0) -- (7,0);
        \end{tikzpicture}
        \caption{{Illustration to the proof of Proposition \ref{thdf22}, Case 2.2.2}}
        \label{wes42123}
\end{figure}
              


Now we are ready to estimate the number of closed factors in $w'$. Combining the cases above, we have $$ \mathtt{Cl}(upux^{(i)})- \mathtt{Cl}(upux^{(i)-})  \leq   \max\left\{\frac{n (\alpha-1)}{\alpha^2} +i, \delta n+i\right\}.$$ So, we can take
 $\delta = \frac{\alpha-1}{ \alpha^2 }$ (or in fact any number smaller than that).
Then, by Lemma \ref{lem0}, we have
$ \mathtt{Cl}(upux^{(i)})- \mathtt{Cl}(upux^{(i)-})  \leq   \frac{n (\alpha-1)}{\alpha^2} +i.$ 
We set $|x|=\xi n $ for some $\xi>0$.
Thus  using Theorem \ref{pth}, we have
\begin{align*}
    \mathtt{Cl}(w') 
                \leq \frac{n^2}{6} +\frac{7n}{6}+1 + \sum_{i=1}^{ \xi n } \left(\frac{n (\alpha-1)}{\alpha^2} +i \right) = \frac{n^2}{6} +\frac{7n}{6}+1 + \frac{\xi n^2 (\alpha-1)}{\alpha^2} +  \frac{\xi^2 n^2}{2} + \frac{\xi n}{2}.            
\end{align*}
Now, since $|w'|=n+|x|= (1+\xi)n$, we have 
\begin{align*}
   & \inf \left\{ \frac{ \mathtt{Cl}(w') }{(1+\xi)^2 n^2} : w'=wx \in Fac(v), |w|=n, |x|=\xi n, exp(w)=\alpha, n \geq N \right\} \\
   & \leq \frac{1}{(1+\xi)^2} \left(\frac{1}{6} + \frac{\xi (\alpha-1)}{\alpha^2} +  \frac{\xi^2 }{2} \right). 
\end{align*}

We let denote the obtained function by $f(\xi,\alpha)$: $$f(\xi, \alpha)= \frac{1}{(1+\xi)^2} \left(\frac{1}{6} + \frac{\xi (\alpha-1)}{\alpha^2} +  \frac{\xi^2}{2}\right).$$ The previous inequality implies by Definition \ref{eq:C_u} that $f(\xi,\alpha)$ is an upper bound for the closed-rich constant $C_v$ of $v$: $C_v\leq f(\xi, \alpha).$ 

So, we can choose for each $\alpha$ the value of $\xi\in(0,+\infty)$ giving minimum of our bound $f(\xi,\alpha)$, and then to get an upper bound for all $\alpha \in (1.4,2)$, take maximum over $\alpha$ in this interval.
(Here we ignore the fact that $\xi n$ can be non-integer, since taking the previous or the next length does not affect the infimum. This follows, for example, from the fact that adding one  letter to a word adds only a sublinear number of closed factors). So,
    $$C_v\leq \max_{\alpha \in (1.4,2)} \min_{\xi \in (0,+\infty)}  f(\xi, \alpha).$$

 For each $\alpha \in (1.4,2)$, considering $f$ as a function of one argument $\xi$, 
 we get that $f$ attains its minimum for $\xi=\xi_0(\alpha) = \frac{\alpha^2 +3 - 3\alpha}{3(\alpha^2 +1 - \alpha)}$.
 
 Now, considering $f(\xi_0(\alpha), \alpha)$ as a function of $\alpha$, with standard arguments one can see that it is growing on the interval $(1.4,2)$ (since its derivative is positive), so its maximum is attained when $\alpha=2$. This gives an upper bound $f(\xi_0(2), 2)= \frac{13}{80}$.  \end{proof}

Similarly to the proof of proof  Proposition \ref{thdf22}, one can prove its symmetric analog:
\begin{proposition}\label{thdf221}
     Let $\alpha$ be the asymptotic critical exponent of an infinite closed-rich word  $v$ such that  it has long $(\alpha-\epsilon)$-powers for any $\epsilon>0$ but it does not have long $\alpha$-powers where  $1.4< \alpha \leq 2$. Then the closed-rich constant is at most $ \frac{13}{80}$.
\end{proposition}

\begin{proposition}\label{thdf1}
    Let $v$ be an infinite closed-rich word such that it has long $\alpha$-powers but it does not have long $\alpha^+$-powers where  $2 \leq \alpha \leq  2.5$. Then the closed-rich constant of $v$ is at most $\frac{967}{5827} $.
\end{proposition}
\begin{proof}
Since $v$  has long $\alpha$-powers but it does not have long $\alpha^+$-powers, then for each $\epsilon >0$, there exists a positive integer $M$ such that the length of any factor of $v$ with exponent $\alpha + \epsilon$ is at most $M$.
For a fixed $\delta >0$, we set $N=\lceil \frac{M}{\delta}\rceil$. Let $w$ be a $\alpha$-power factor of $v$ of length $n\geq N$ (such $w$ exists in $v$ by the condition of the proposition). Note that for each $(\alpha+\epsilon)$-power factor $z$ of $w$, 
where $\epsilon>0$, its length is at most $M\leq \delta N\leq \delta n $; we will use this fact several times in the proof.
Then $w$ is of the form $w=uuu^{(1)}$ where $u \in \Sigma^+$, $u^{(1)} \in Pref(u)$, $|u|=l$, $|u^{(1)}|=r$, and $\frac{2l+r}{l}=\alpha$. Since $|w|=2l+r=n$, we have $l=\frac{n}{\alpha}$ and $r=\frac{n(\alpha-2)}{\alpha}$.
Consider some occurrence of $w$ in $v$ and its extension to the right by a factor of length $\xi n$ (we will choose $\xi$ later). 
We let $x$ denote this extension and $w'$ be the extended factor, i.e.,  $w'=uu u^{(1)} x \in Fac(v)$ where $|x|=\xi n$. 
Let $u=u^{(1)} a u^{(2)}$ and $x=bx'$ where $a, b \in \Sigma$, $u^{(2)}, x' \in \Sigma^*$. Since $w$ is an $\alpha$-power and not an $\alpha^+$-power, we have $w_{l+r+1} \neq w_{2l+r+1}$, i.e., $a \neq b$.
Let $x^{(i)} \in \Sigma^+$ be a prefix of $x$ of length $i$, $t^{(i)}$ be the largest repeated suffix of $uu u^{(1)} x^{(i)}$ and $z^{(i)}$ be the largest repeated suffix of $t^{(i)}$. 
If $t^{(i)} = \lambda$, then $\mathtt{Cl}(uuu^{(1)} x^{(i)})- \mathtt{Cl}(uuu^{(1)} x^{(i)-}) = 1$.
Now, consider $t^{(i)} \in \Sigma^+$. 
Let us denote the rightmost occurrence of $t^{(i)}$ in $uu u^{(1)} x^{(i)}$ by $t^{(i1)}$ and the rightmost occurrence  $t^{(i)}$ in $uu u^{(1)} x^{(i)-}$ by $t^{(i2)}$. 

We first show that $|t^{(i)}| \leq l+i-1$. Assume the converse: suppose that $|t^{(i)}| \geq l+i$. Since $|u|=l$ and $uuu^{(1)}x^{(i)}= u^{(1)} a u^{(2)} u^{(1)} a u^{(2)} u^{(1)} x^{(i)}$, we have $t^{(i1)} = t' a u^{(2)} u^{(1)} x^{(i)}$ for some $t' \in \Sigma^*$. As $b \in Pref(x^{(i)})$, $t' a u^{(2)} u^{(1)} b$ is a prefix of $t^{(i1)}$, i.e., $t' a u^{(2)} u^{(1)} b \in Pref(t^{(i2)})$. Then as $t^{(i2)}$ is a factor of $uuu^{(1)}x^{(i)-}$, $t' a u^{(2)} u^{(1)} b$ is a factor of $uuu^{(1)}$. Since $|a u^{(2)} u^{(1)} b|=l+1$ and $l$ is the period of $uuu^{(1)}$, $t' a u^{(2)} u^{(1)} b \in Fac(uuu^{(1)})$ implies $a=b$, which is a contradiction. Thus, $|t^{(i)}| \leq l+i-1$. 

Now, if $|t^{(i1)}| \leq r+i$, then we will use this bound for $|t^{(i)}|$. 
Consider now the case $|t^{(i1)}| > r+i$. Due to this inequality and the condition $|t^{(i)}| \leq l+i-1$, the factor $t^{(i1)}$ begins inside the second occurrence of $u$ (i.e., between positions $l$ and $2l$). We let $g$ denote the factor which occurs as the intersection of $t^{(i1)}$ and the second occurrence of $u$, i.e., $g$ is a suffix of $u$ and a prefix of $t^{(i1)}$.
Now, we have two cases depending on whether $t^{(i1)}$ and $t^{(i2)}$ intersect or not:
\begin{itemize}
    \item \textbf{Case 1} : $t^{(i1)}$ and $t^{(i2)}$ do not intersect. \\
    Then $t^{(i2)} y t^{(i1)} \in Suff(uu u^{(1)} x^{(i)})$ for some $y \in \Sigma^*$. Now, we have the following cases:
    \begin{itemize}
        \item \textbf{Case 1.1.}: $|t^{(i2)} y t^{(i1)} | \leq l+r+i$.\\
        Then $2|t^{(i)}| \leq l+r+i$, i.e., $|t^{(i)}| \leq \frac{l+r+i}{2}$, and we use this bound in this case.

        \item \textbf{Case 1.2.}: $|t^{(i2)} y t^{(i1)} | > l+r+i$. \\
        Then $t^{(i2)}$ has a prefix, say $p \in \Sigma^+$, which is also a suffix of $u$. Now, $|p| \neq |g|$ as $u^{(1)} b$ is not a prefix of $u$. Now,      
        we have the following three cases:

        \begin{itemize}
            \item \textbf{Case 1.2.1.}: If $|g|< |p| < 2|g|$ (See Fig. \ref{vgvg}):\\
            Then, as $g \in Suff(u) \cap Pref(t_{i2})$, both of these $g$ intersect with each other, i.e., $g=s^{(1)} s^{(2)} = s^{(2)} s^{(3)}$ for some $s^{(1)}, s^{(2)}, s^{(3)} \in \Sigma^+$. This implies $s^{(1)} = p^{(1)} q^{(1)}$, $s^{(2)} = (p^{(1)} q^{(1)})^j p^{(1)}$ and $s^{(3)} = q^{(1)} p^{(1)}$ for some $j \geq 0$, $p^{(1)} \in \Sigma^+$ and $q^{(1)} \in \Sigma^*$. 
            It follows that $p^{(1)} q^{(1)} g \in Suff(u)$ and $g q^{(1)} p^{(1)} \in Pref(t^{(i2)})$. 
            Then we have $p^{(1)} q^{(1)} g q^{(1)} p^{(1)} = (p^{(1)} q^{(1)})^{j+3} p^{(1)} \in Fac(uu^{(1)} x^{(i)})$ (the factor starting at position $2l-|p|+1$ in $w'$). 
            Thus, as the exponent of $p^{(1)} q^{(1)} g q^{(1)} p^{(1)}$ is greater than $\alpha+ \epsilon$, 
            we have $|p^{(1)} q^{(1)} g q^{(1)} p^{(1)}| < n \delta$, i.e, $|g| < n\delta$. Therefore, as $|t^{(i1)}| = |g|+r+i $, we have  $|t^{(i)}| < n \delta + r+i $.

\begin{figure}[h]
    \centering
        \begin{tikzpicture}
                \draw (2,0) arc[start angle=0, end angle=180, radius=1] node[midway, above] {$u, l$};
                \draw (4,0) arc[start angle=0, end angle=180, radius=1] node[midway, above] {$u, l$};
                \draw (5.4,0) arc[start angle=0, end angle=180, radius=.7] node[midway, above] {$u^{(1)},~ r$};
                \draw (5.7,0) arc[start angle=0, end angle=-180, radius=1.2] node[midway, above] {$t^{(i1)}$};
                \draw[<->] (3.4,0.15) -- (3.9, 0.15) node[midway, above] {$g$};
                \draw (3.1,0) arc[start angle=0, end angle=-180, radius=1.2] node[midway, above] {$t^{(i2)}$};
                \draw[<->] (1.3,0.15) -- (1.9, 0.15) node[midway, above] {$g$};
                \draw[<->] (0.8,0.3) -- (1.5, 0.3) node[midway, above] {$g$};
                \draw[<->] (0.8, -0.15) -- (1.9,- 0.15) node[midway, below] {$p$};
                \draw[-] (0,0) -- (6.8,0);
                \draw[<->] (5.4,0.15) -- (5.7,0.15)node[right, above] {$x^{(i)}$};
                \end{tikzpicture}
                \caption{{Illustration to the proof of Proposition \ref{thdf1}, Case 1.2.1 }}
                \label{vgvg}
    \end{figure}

    \begin{figure}[h]
        \centering
        \begin{tikzpicture}
                \draw (2,0) arc[start angle=0, end angle=180, radius=1] node[midway, above] {$u, l$};
                \draw (4,0) arc[start angle=0, end angle=180, radius=1] node[midway, above] {$u, l$};
                \draw (5.4,0) arc[start angle=0, end angle=180, radius=.7] node[midway, above] {$u^{(1)},~ r$};
                \draw (5.7,0) arc[start angle=0, end angle=-180, radius=1.2] node[midway, above] {$t^{(i1)}$};
                \draw[<->] (3.4,0.15) -- (3.9, 0.15) node[midway, above] {$g$};
                \draw (3.1,0) arc[start angle=0, end angle=-180, radius=1.3] node[midway, above] {$t^{(i2)}$};
                \draw[<->] (1.4,0.15) -- (1.9, 0.15) node[midway, above] {$g$};
                \draw[<->] (0.6,0.15) -- (1.1, 0.15) node[midway, above] {$g$};
                \draw[<->] (1.112,0.15) -- (1.3555, 0.15) node[midway, above] {$y'$};
                \draw[<->] (0.6, -0.15) -- (1.9,- 0.15) node[midway, below] {$p$};
                \draw[<->] (2.1,0.15) -- (3, 0.15) node[midway, above] {$h$};
                \draw[-] (0,0) -- (6.8,0);
                \draw[<->] (5.4,0.15) -- (5.7,0.15)node[right, above] {$x^{(i)}$};
                \end{tikzpicture}
                \caption{{Illustration to the proof of Proposition \ref{thdf1}, Case 1.2.2 }}
                \label{vgvg2}
\end{figure}

            
            \item \textbf{Case 1.2.2.} : If $|p| \geq 2|g|$ (see Fig. \ref{vgvg2}):
            
            Then $p= g y' g$ and $t^{(i2)} = g y' g h$ for some $y' \in \Sigma^*$ and $ h \in \Sigma^+$. 

            If $|y'| >r$, then as $gy'g \in Pref(t^{(i1)})$, we have $|g| \leq i$. Indeed, since $gy'g \in Pref(t^{(i1)})$ and $|y'| >r$, we have $u^{(1)} \in Pref(y')$. Then the last $g$ in $gy'g$ lies in $x^{(i)}$. This gives $|g| \leq i$. So, since $|t^{(i1)}| = |g|+r+i $, we have  $|t^{(i)}| \leq r+ 2i $. 
            
            Now, consider $|y'| \leq r$.
            \begin{itemize}
                \item If $|h|>r$, then as $y'gh \in Suff(t^{(i1)})$ and $|y'gh|=r+i$, we have $|g|<i$.  Then as $|t^{(i1)}| = |g|+r+i $, we have $|t^{(i)}| \leq r+ 2i $.
                
                \item  If $|h|\leq r$, then $h \in Pref(u^{(1)})$ and $u^{(1)} \in Pref(y'gh)$. Indeed, $u^{(1)}$ occurs in $t^{(i1)}$ at position $|g|+1$, then it also occurs in $t^{(i2)}$ at the same position $|g|+1$, and we also have there an occurrence of $y'gh$, which is longer than $u^{(1)}$, so  $u^{(1)}$ is a prefix of $y'gh$. Since $h\in Pref(u^{(1)})$, we also have $h \in Pref(y'gh)$). Since $t^{(i2)} = gy'gh$, this implies that $|z^{(i)}| \geq |gh|$. Thus, $|t^{(i)}|-|z^{(i)}| \leq |gy'| \leq r+i$.

            \end{itemize}


            \item \textbf{Case 1.2.3.} : If $|p|<|g|$ (See Fig. \ref{vgvg3}): 
            
            Then as $g \in Suff(u)$ and $g \in Pref(t^{(i2)})$, the occurrence of $g$ as a suffix of the first occurrence of $u$ and the occurrence of $g$ as a prefix of $t^{(i2)}$ intersect. 
     Then $p$ is their intersection and  $g= h^{(1)} p = p h^{(3)}$ for some $h^{(1)}, h^{(3)} \in \Sigma^+$. This implies $h^{(1)} = s q$, $p = (s q)^j s$ and $h^{(3)} = q s$ for some $j \geq 0$, $s\in \Sigma^+$ and $q \in \Sigma^*$.
            If $j \geq 1$, then as exponent of $h^{(1)} p h^{(3)}$ is greater than $\alpha + \epsilon$, 
            we have $|h^{(1)} p h^{(3)}| <  \delta n$, i.e., $|g|<n\delta$, i.e.,  $|t^{(i)}| < n\delta+ r+ i $.
            Now suppose that $j=0$ and consider the following cases:

            \begin{itemize}
                \item \textbf{Case 1.2.3.1.} :  If $|h^{(1)}| \leq\frac{|g|}{1.5+\epsilon}$, then using $g=sqs$, we have $\frac{|h^{(1)} p h^{(3)}|}{|h^{(1)}|} \geq 2.5+\epsilon$, i.e.,                
                $h^{(1)} p h^{(3)}= h^{(1)} g =(sq)^\beta$ where $\beta \geq 2.5+\epsilon$. Then  $|h^{(1)} p h^{(3)}| <  \delta n$, i.e., $|g|<n\delta$. Thus, $|t^{(i)}| < n\delta+ r+ i $, i.e., $|t^{(i)}|-|z^{(i)}| < n\delta +r+i$.
                
                \item \textbf{Case 1.2.3.2.} : If $|h^{(1)}| > \frac{|g|}{1.5+\epsilon}$, then as $2|t^{(i)}| \leq 2l +r +i -(l-|g|+|h^{(1)}|)$, we have $2|t^{(i)}| \leq l +r +i + |g|-\frac{|g|}{1.5+\epsilon}$. This gives $2|t^{(i)}| \leq l +r +i + \frac{(0.5+\epsilon)|g|}{1.5+\epsilon}$. Using $|g|=|t^{(i)}|-r-i$, we have 
                  $|t^{(i)}| \leq \frac{1.5+\epsilon}{2.5+\epsilon} l + \frac{r+i}{2.5+\epsilon}$.            
            \end{itemize}


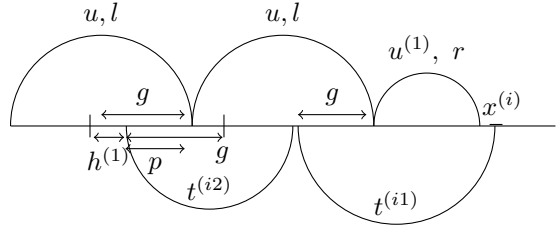
\begin{figure}[h]
    \centering
        \begin{tikzpicture}
                \draw (2,0) arc[start angle=0, end angle=180, radius=1.2] node[midway, above] {$u, l$};
                \draw (4.4,0) arc[start angle=0, end angle=180, radius=1.2] node[midway, above] {$u, l$};
                \draw (5.8,0) arc[start angle=0, end angle=180, radius=.7] node[midway, above] {$u^{(1)},~ r$};
                \draw (6,0) arc[start angle=0, end angle=-180, radius=1.3] node[midway, above] {$t^{(i1)}$};
                \draw[<->] (3.4,0.15) -- (4.3, 0.15) node[midway, above] {$g$};
                \draw (3.33,0) arc[start angle=0, end angle=-180, radius=1.1] node[midway, above] {$t^{(i2)}$};
                \draw[<->] (.8,0.15) -- (1.9, 0.15) node[midway, above] {$g$};
                 \draw[] (0.65,0.15) -- (0.65, -0.15);
                 \draw[] (2.42,0.15) -- (2.42, -0.15);
                \draw[<->] (1.12, -0.3) -- (1.9,- 0.3) node[midway, below] {$p$};
                \draw[<->] (1.13, -0.15) -- (2.4,- 0.15) node[right, below] {$g$};
                \draw[<->] (0.7, -0.15) -- (1.1,- 0.15) node[midway, below] {$h^{(1)}$};
                \draw[-] (-0.4,0) -- (6.8,0);
                \draw[] (5.93,0.02) -- (6.1,0.02)node[right, above] {$x^{(i)}$};
                \end{tikzpicture}
                \caption{{Illustration to the proof of Proposition \ref{thdf1}, Case 1.2.3 }}
                \label{vgvg3}
\end{figure}

\begin{figure}[h]
        \centering
        \begin{tikzpicture}
                \draw (2,0) arc[start angle=0, end angle=180, radius=1.2] node[midway, above] {$u, l$};
                \draw (4.4,0) arc[start angle=0, end angle=180, radius=1.2] node[midway, above] {$u, l$};
                \draw (5.8,0) arc[start angle=0, end angle=180, radius=.7] node[midway, above] {$u^{(1)},~ r$};
                \draw (6,0) arc[start angle=0, end angle=-180, radius=1.3] node[midway, above] {$t^{(i1)}$};
                \draw[<->] (3.4,-0.15) -- (4.3, -0.15) node[midway, below] {$g$};
                \draw[<->] (3.35,0.15) -- (3.6, 0.15) node[midway, above] {$f^{(2)}$};
                \draw (3.6,0) arc[start angle=0, end angle=-180, radius=1.3] node[midway, above] {$t^{(i2)}$};
                \draw[<->] (.7,0.15) -- (1.9, 0.15) node[midway, above] {$g$};
                 \draw[] (0.65,0.15) -- (0.65, -0.15);
                 \draw[] (2.42,0.15) -- (2.42, -0.15);
                \draw[<->] (1.12, -0.3) -- (1.9,- 0.3) node[midway, below] {$p$};
                \draw[<->] (1.13, -0.15) -- (2.4,- 0.15) node[right, below] {$g$};
                \draw[<->] (0.7, -0.15) -- (1,- 0.15) node[midway, below] {$h^{(1)}$};
                \draw[-] (-0.4,0) -- (6.8,0);
                \draw[] (5.93,0.02) -- (6.2,0.02)node[midway, above] {$x^{(i)}$};
                \end{tikzpicture}
                \caption{{Illustration to the proof of Proposition \ref{thdf1}, Case 2.2.3 }}
                \label{vgvg23}
\end{figure}
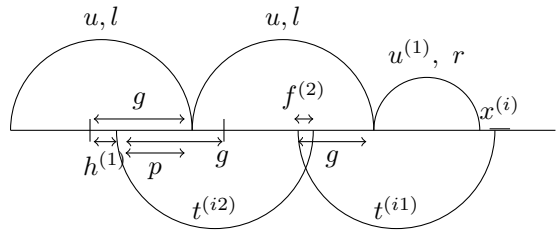


          
        \end{itemize}
    \end{itemize}

        \item \textbf{Case 2} : {Let $t^{(i1)}$ and $t^{(i2)}$ intersect}.
        
        Then, $t^{(i2)}=f^{(1)} f^{(2)}$ and $t^{(i1)}=f^{(2)} f^{(3)}$ for some $f^{(1)}, f^{(2)}, f^{(3)} \in \Sigma^+$. Then $|z^{(i)}| \geq |f^{(2)}|$. Now, we have the following cases:
        \begin{itemize}
            \item \textbf{Case 2.1.} : Let $|f^{(1)} f^{(2)} f^{(3)}| \leq l+r+i $. \\
            Then $2|t^{(i)}| \leq l+r+i +|f^{(2)}|$, i.e., $|t^{(i)}| \leq \frac{l+r+i}{2} + \frac{|f^{(2)}|}{2}$. Since  $|z^{(i)}| \geq |f^{(2)}|$, $|t^{(i)}|-|z^{(i)}| \leq \frac{l+r+i}{2} + \frac{|f^{(2)}|}{2} -|f^{(2)}|$, i.e., $|t^{(i)}|-|z^{(i)}| \leq \frac{l+r+i}{2}$.
            
            \item \textbf{Case 2.2.} : Let $|f^{(1)} f^{(2)} f^{(3)}| > l+r+i $ (We will proceed similarly to Case 1.2.). \\
            Then $t^{(i2)}$ has a prefix, say $p \in \Sigma^+$, which is also a suffix of $u$. Now, $|p| \neq |g|$ as $u^{(1)} b$ is not a prefix of $u$. Now, we have the following three cases:
                \begin{itemize}
                    \item \textbf{Case 2.2.1.}: If $|g|< |p| < 2|g|$ :\\
                    Then, similarly to Case 1.2.1., we have $|t^{(i)}|-|z^{(i)}| \leq n\delta + r+i$.

                     \item \textbf{Case 2.2.2.}: If $|p| \geq 2|g|$ : \\
                     Then,  similarly to Case 1.2.2.,   we have  $|t^{(i)}|-|z^{(i)}| \leq r+2i$.

                      \item \textbf{Case 2.2.3.}: If $|p|<|g|$ (See Fig. \ref{vgvg23}): 
                      
                      Then the occurrence of $g$ as a suffix of the first occurrence of $u$ and the occurrence of $g$ as a prefix of $t^{(i2)}$ intersect. Then $p$ is their intersection and $g= h^{(1)} p = p h^{(3)}$ for some $h^{(1)}, h^{(3)} \in \Sigma^+$. This implies $h^{(1)} = s q$, $p = (sq)^js$ and $h^{(3)} = qs$ for some $j \geq 0$, $s\in \Sigma^+$ and $q \in \Sigma^*$.
                      If $j \geq 1$, then similarly to Case 1.2.3., we have $|t^{(i)}|-|z^{(i)}| < n\delta +r+i$. In the case $j=0$  we consider the following cases:

            \begin{itemize}
                \item \textbf{Case 2.2.3.1.} :  If $|h^{(1)}| \leq\frac{|g|}{1.5+\epsilon}$, 
                then using $g=sqs$, we have $\frac{|h^{(1)} p h^{(3)}|}{|h^{(1)}|} \geq 2.5+\epsilon$, i.e., $h^{(1)} p h^{(3)}= h^{(1)} g =(sq)^\beta$ where $\beta \geq 2.5+\epsilon$. Then  $|h^{(1)} p h^{(3)}| <  \delta n$, i.e., $|g|<n\delta$. Thus, $|t^{(i)}| < n\delta+ r+ i $.
                
                \item \textbf{Case 1.2.3.2.} : If $|h^{(1)}| > \frac{|g|}{1.5+\epsilon}$, then as $2|t^{(i)}| \leq 2l +r +i -(l-|g|+|h^{(1)}|)+|f^{(2)}|$, we have $2|t^{(i)}| \leq l +r +i + |g|-\frac{|g|}{1.5+\epsilon}+|f^{(2)}|$. This gives $2|t^{(i)}| \leq l +r +i + \frac{(0.5+\epsilon)|g|}{1.5+\epsilon}+|f^{(2)}|$. 
                Using $|g|=|t^{(i)}|-r-i$, we have 
                $|t^{(i)}| \leq  \frac{1.5+\epsilon}{2.5+\epsilon} l + \frac{r+i}{2.5+\epsilon}  + \frac{1.5+\epsilon}{2.5+\epsilon} |f^{(2)}| $. 
                Then, as $|z^{(i)}|\geq |f^{(2)}|$, we have  $|t^{(i)}| - |z^{(i)}| \leq  \frac{1.5+\epsilon}{2.5+\epsilon} l + \frac{r+i}{2.5+\epsilon}  + \frac{1.5+\epsilon}{2.5+\epsilon} |f^{(2)}| - |f^{(2)}|$, 
                i.e.,   $|t^{(i)}| - |z^{(i)}| \leq  \frac{1.5+\epsilon}{2.5+\epsilon} l + \frac{r+i}{2.5+\epsilon}$.           
            \end{itemize}


                \end{itemize}
        
        \end{itemize}
    
\end{itemize}

Therefore, combining all previous cases for $t^{(i)} \in \Sigma^+$, we have:


\[ |t^{(i)}|-|z^{(i)}|  \leq  
    \max\left\{r+ 2i,  \frac{l+r+i}{2}, n \delta +r +i,       \frac{(1.5+\epsilon)l}{2.5+\epsilon} + \frac{r}{2.5+\epsilon} + \frac{i}{2.5+\epsilon}\right\}
\]

 Analyzing these four expressions, one can see that if  $|x| \leq \frac{3(l-r)}{8}$, then   the fourth one gives the maximum, and if
$\frac{3(l-r)}{8} < |x| \leq l-r$, then the first one gives the maximum (under the condition that $\delta$ is small enough and taking into account that we can take $\epsilon$ is as small as needed). Consider $\frac{3(l-r)}{8} < |x| \leq l-r$ (note that $|x| \leq \frac{3(l-r)}{8}$ gives a worse bound, which was reported at WORDS conference  \cite{DBLP:conf/cwords/MaityP25}).

So, applying Theorem \ref{pth} for $w$ and adding new closed factors ending at positions inside $x$, we obtain
\begin{align*}
    \mathtt{Cl}(w') 
                 &\leq \frac{n^2}{6} +\frac{7n}{6}+1 + \sum_{i=1}^{\frac{3(l-r)}{8}}  \left(\frac{(1.5+\epsilon)l}{2.5+\epsilon} + \frac{r}{2.5+\epsilon}  + \frac{i}{2.5+ \epsilon} \right) +  \sum_{i=\frac{3(l-r)}{8}+1}^{\xi n} (r+2i)\end{align*}

Summing arithmetic progressions, expressing $l$ and $r$ in terms of $n$ and $\alpha$ and ignoring linear terms, we obtain 

\begin{align*}
    \mathtt{Cl}(w')                  &\leq 
                                          \frac{n^2}{6}  + 
                     \frac{3(1.5+\epsilon)(3-\alpha)n^2}{8(2.5+\epsilon)\alpha^2}  
                      + \frac{3 (\alpha-2) (3-\alpha)n^2}{8(2.5+\epsilon)\alpha^2} 
                        + \frac{9 (3-\alpha)^2 n^2}{128(2.5+\epsilon)\alpha^2} \\ & +
                        \frac{\xi (\alpha-2) n^2}{\alpha}  
                         - \frac{3 (\alpha-2) (3-\alpha) n^2}{8\alpha^2} +
                        \xi^2 n^2 
                         - \frac{9(3-\alpha)^2 n^2}{64\alpha^2}. \\ 
\end{align*}


Recall that we have $|w'|=n+|x|= (1+\xi)n$. Using Definition \ref{eq:C_u} and taking into account the fact that we can choose $\epsilon$ as small as we want, we obtain the following upper bound for the closed-rich constant $C_v$ of $v$: 

\begin{align*}
   C_v\leq &\inf \left\{ \frac{ \mathtt{Cl}(w') }{|w'|^2 } : w'=wx \in Fac(v), |w|=n, |x|=\xi n,  exp(w)=\alpha, n \geq N \right\} \\
& \leq \frac{1}{(1+\xi)^2} \left( \frac{1}{6} +
                        \frac{\xi (\alpha-2) }{\alpha}  +
                        \xi^2    + 
                     \frac{9(3-\alpha)^2}{80\alpha^2} \right).
\end{align*}

For convenience, we rewrite it using $t=\frac{1}{\alpha}$ and $z=1+\xi$:
$$C_v \leq \frac{1}{z^2} \left( \frac{1}{6} + (z-1)(1-2t) + (z-1)^2 + \frac{9(3t-1)^2}{80}  \right).$$

We let denote the latter function by $f(t,z)$. To get an upper bound for $C_v$, we need to find $\max_t \min_z f(t,z)$, under the conditions $2 \leq \alpha \leq 2.5$ and $ \frac{3(l-r)}{8} < \xi n \leq l-r$, so in terms of $z$ and $t$ we have $0.4 \leq t \leq 0.5$ and $ \frac{9t+5}{8} < z \leq 3t$:

   $$C_v \leq \max_{t \in [0.4,0.5]} \min_{z \in (\frac{9t+5}{8}, 3t]}  f(t,z).$$
   
Analyzing the derivative $f_z(t,z)$, we get that for each $t$ the function  $f(t,z)$ reaches its global minimum when $z=z_0(t)=\frac{67 +318 t + 243t^2}{120(1+2t)}$. Therefore, for each $t$, the minimum of $f(t,z)$ in the interval $z \in (\frac{9t+5}{8}, 3t]$ is $\min \{f(t,z_0(t)), f(t, \frac{9t+5}{8}), f(t,3t) \}$. Since for $t \in [0.4, 0.5]$ we have  $f(t,z_0(t)) < f(t, \frac{9t+5}{8}) < f(t,3t)$, then 
 $f(t, z_0(t))$ gives the minimum. We now need the maximum value of $f(t,z_0(t))$. The derivative $f'(t,z_0(t))$ is negative in the interval $[0.4,0.5]$, so $f(t,z_0(t))$ attains its maximum value at $t=0.4$. This gives an upper bound $f(0.4, z_0(0.4))= \frac{967}{5827} $. 
\end{proof}

Using proof techniques similar to the proof of Proposition \ref{thdf1}, we can prove its analog for ($\alpha-\varepsilon$)-powers:

\begin{proposition}\label{qa11}
     Let $\alpha$ be the asymptotic critical exponent of an infinite closed-rich word  $v$ such that   it has long $(\alpha-\epsilon)$-powers for any $\epsilon>0$ but it does not have long $\alpha$-powers where  $2< \alpha \leq 2.5$. Then, the closed-rich constant is at most $\frac{967}{5827}$.
\end{proposition}
\begin{proof} 
  Since $v$  has long $(\alpha-\epsilon)$-powers for any $\epsilon>0$ but it does not have long $\alpha^*$-powers, then for each $\eta \geq0$, there exists a positive integer $M$ such that the length of any factor of $v$ with exponent $ \alpha + \eta$ is at most $M$.
For a fixed $\delta >0$, we set $N=\lceil \frac{M}{\delta}\rceil$. 
Let $w$ be a $(\alpha-\epsilon_1)$-power factor of $v$ of length $n\geq N$ where $\epsilon_1 > 0$ and $(\alpha-\epsilon_1)>2$ (such $w$ exists in $v$ by the condition of the proposition).
Note that for each $(\alpha+\eta)$-power factor $z$ of $w$, 
where $\eta \geq 0$, its length is at most $M\leq \delta N\leq \delta n $; 
Consider $\alpha -\epsilon_1 = \alpha_{\epsilon_1}$; so $2 < \alpha_{\epsilon_1} < 2.5$.
Then $w$ is of the form $w=uuu^{(1)}$ where $u \in \Sigma^+$, $u^{(1)} \in Pref(u)$, $|u|=l_{\epsilon_1}$, $|u^{(1)}|=r_{\epsilon_1}$, and $\frac{2l_{\epsilon_1}+r_{\epsilon_1}}{l_{\epsilon_1}}= \alpha_{\epsilon_1}$.
Since $|w|=2l_{\epsilon_1}+r_{\epsilon_1}=n$, we have  $l_{\epsilon_1}=\frac{n}{\alpha_{\epsilon_1}}$ and $r_{\epsilon_1}=\frac{n (\alpha_{\epsilon_1}-2)}{\alpha_{\epsilon_1}}$.
Let $l=\frac{n}{\alpha}$ and $r=\frac{n(\alpha-2)}{\alpha}$ where $2 < \alpha \leq 2.5$.

Consider some occurrence of $w$ in $v$ and its extension to the right by a factor of length $\xi n$ (we will choose $\xi$ later). 
We let $x$ denote this extension and $w'$ be the extended factor, i.e.,  $w'=uu u^{(1)} x \in Fac(v)$ where $|x|=\xi n$.  
Let $u=u^{(1)} a u^{(2)}$ and $x=bx'$ where $a, b \in \Sigma$, $u^{(2)}, x' \in \Sigma^*$. Since $w$ is an $\alpha_{\epsilon_1}$-power and not an $\alpha_{\epsilon_1}^+$-power, we have $w_{l_{\epsilon_1}+r_{\epsilon_1}+1} \neq w_{2l_{\epsilon_1}+r_{\epsilon_1}+1}$, i.e., $a \neq b$.
Let $x^{(i)} \in \Sigma^+$ be a prefix of $x$ of length $i$, $t^{(i)}$ be the largest repeated suffix of $uu u^{(1)} x^{(i)}$ and $z^{(i)}$ be the largest repeated suffix of $t^{(i)}$. 
If $t^{(i)} = \lambda$, then $\mathtt{Cl}(uuu^{(1)} x^{(i)})- \mathtt{Cl}(uuu^{(1)} x^{(i)-}) = 1$.
Now, consider $t^{(i)} \in \Sigma^+$. 
Then, similar to the Proposition \ref{thdf1}, if  $|x| \leq \frac{3(l_{\epsilon_1}-r_{\epsilon_1})}{8}$, $|t^{(i)}| - |z^{(i)}|   \leq  \frac{(1.5+\eta)l_{\epsilon_1}}{2.5+\eta} + \frac{r_{\epsilon_1}}{2.5+\eta}  + \frac{i}{2.5+\eta} $; if $\frac{3(l_{\epsilon_1}-r_{\epsilon_1})}{8} < |x| \leq l_{\epsilon_1}-r_{\epsilon_1}$, $|t^{(i)}| - |z^{(i)}|   \leq r_{\epsilon_1} + 2i$
(under the condition that $\delta$ is small enough and taking into account that we can take $\epsilon_1$ is as small as needed). Consider $\frac{3(l_{\epsilon_1}-r_{\epsilon_1})}{8} < |x| \leq l_{\epsilon_1}-r_{\epsilon_1}$.


Now,
$r_{\epsilon_1} + 2i < r +2i$, and for $0 \leq \eta <0.5$,
$\frac{(1.5+\eta)l_{\epsilon_1}}{2.5+\eta} + \frac{r_{\epsilon_1}}{2.5+\eta}  + \frac{i}{2.5+\eta} \leq \frac{(1.5+\eta)l}{2.5+\eta} + \frac{r}{2.5+\eta}  + \frac{i}{2.5+\eta}.$

Therefore under the above assumptions on $|x|$, $\delta$ and $\eta$, Lemma \ref{lem0} gives,  for $|x| \leq \frac{3(l_{\epsilon_1}-r_{\epsilon_1})}{8}$,
$$\mathtt{Cl}(uuu^{(1)} x^{(i)})- \mathtt{Cl}(uu u^{(1)} x^{(i)-}) \leq  \frac{(1.5+\eta)l}{2.5+\eta} + \frac{r}{2.5+\eta}  + \frac{i}{2.5+\eta};$$ 
and for  $\frac{3(l_{\epsilon_1}-r_{\epsilon_1})}{8} < |x| \leq l_{\epsilon_1}-r_{\epsilon_1}$,
$$\mathtt{Cl}(uuu^{(1)} x^{(i)})- \mathtt{Cl}(uu u^{(1)} x^{(i)-}) \leq r+2i.$$
Then following the similar steps of Proposition \ref{thdf1}, we can show that the closed-rich constant is at most $\frac{967}{5827}$.

 \end{proof}

\begin{proposition}\label{cont1}
    If an infinite closed-rich word contains long factors with exponent $\alpha$ where $\alpha > 2.5$, then its closed-rich constant is at most $ \frac{8}{49}$.
\end{proposition}
\begin{proof} 
Let $v$ be an infinite closed-rich word such that it has long factors with exponent $\alpha$ where $\alpha > 2.5$. We distinguish between two cases depending on $\alpha$.
\begin{itemize}
    \item \underline{Case $2.5 < \alpha \leq 3$}. Let  $w=u^{(1)} uu x$ be a factor of $v$  of length $n$ such that 
    $a \in Pref(u)$, $b \in Pref(x)$,  $a,b \in \Sigma$, $a \neq b$,   $u^{(1)} \in \Sigma^+ \cap Suff(u)$, $|u|=|x|=k$, $|u^{(1)}|=k_1 $ and $exp(u^{(1)}uu) = \alpha$, i.e., $\frac{2k+k_1}{k} = \alpha$.\\
Let $x^{(i)}$ be a prefix of $x$ of length $i$, $t^{(i)}$ be the largest repeated suffix of $u^{(1)}uux^{(i)}$ and $z^{(i)}$ be the largest repeated suffix of $t^{(i)}$. 
If $t^{(i)} = \lambda$, then $\mathtt{Cl}(u^{(1)}uu x^{(i)})- \mathtt{Cl}(u^{(1)}uu x^{(i)-}) = 1 \leq k$. Now, consider $t^{(i)} \in \Sigma^+$.
 Let us denote the rightmost $t^{(i)}$ in $u^{(1)}uux^{(i)}$ by $t^{(i1)}$ and the rightmost $t^{(i)}$ in $u^{(1)}uux^{(i)-}$ by $t^{(i2)}$.
 
We first show that $|t^{(i)}| \leq k+i-1$. Assume the converse: suppose that $|t^{(i)}| \geq k+i$. Since $|u|=k$, we have $t^{(i1)} = t' u x^{(i)}$ for some $t' \in \Sigma^*$. As $b \in Pref(x^{(i)})$, $t' u b$ is a prefix of $t^{(i1)}$, i.e., $t' u b \in Pref(t^{(i2)})$. Then as $t^{(i2)}$ is a factor of $u^{(1)}uux^{(i)-}$, $t' u b$ is a factor of $u^{(1)}uu$. Since $a \in Pref(u)$, $|u b|=k+1$ and $k$ is the period of $u^{(1)} uu$, then $t' u b \in Fac(u^{(1)}uu)$ implies $a=b$, which is a contradiction. Thus, $|t^{(i)}| \leq k+i-1$. 
 

Now, if $|t^{(i)}|\leq k$, then $|t^{(i)}|-|z^{(i)}| \leq k$. We will use this bound for $|t^{(i)}|-|z^{(i)}|$.
Consider now the case $k < |t^{(i)}| \leq k+i-1$.
If $t^{(i2)}$ lies in $u^{(1)}uu$, then as period of $u^{(1)}uu$ is $k$, $|z^{(i)}| \geq |t^{(i)}|-k$. This implies $|t^{(i)}|-|z^{(i)}| \leq k$. 
If $t^{(i2)}$ does not lie in $u^{(1)}uu$ but lies in $uux^{(i)}$ (see Fig. \ref{yu1}), then as $k < |t^{(i)}| \leq k+i-1$,
$t^{(i2)}$ has a non-empty suffix which is also a prefix of $x^{(i)}$. Also, $k < |t^{(i)}| \leq k+i-1$ implies 
$t^{(i1)}$ has a non-empty prefix which is also a suffix of second occurrence of $u$. Thus,  
$t^{(i1)}$ and $t^{(i2)}$ intersect, i.e., $t^{(i2)}=s^{(3)} s^{(1)}$ and $t^{(i1)}=s^{(1)}s^{(2)}$ for some $s^{(1)}, s^{(2)}, s^{(3)} \in \Sigma^+$ with $|s^{(2)}| < i \leq k$.
Then by Lemma \ref{4lk}, $s^{(3)}=pq$, $s^{(1)}=(pq)^jp$ and $s^{(2)}=qp$ for some $p \in \Sigma^+, q \in \Sigma^*$, $j \geq 0$.
This implies $t^{(i)}=(pq)^{j+1}p$  and $|pq| < k$. Thus, $|z^{(i)}| \geq (pq)^{j}p$. This implies $|t^{(i)}|-|z^{(i)}| \leq |pq| < k$.
Therefore using Lemma \ref{lem0}, for each $1\leq i \leq k$, $\mathtt{Cl}(u^{(1)}uu x^{(i)})- \mathtt{Cl}(u^{(1)}uu x^{(i)-}) \leq k$.
Then using Theorem \ref{pth} for the prefix of $w$ of length $(n-\frac{k}{2})$ 
and Lemma  \ref{lem0} for each $\frac{k}{2}+1\leq i \leq k$, we get the following upper bound:

\begin{align*}
    \mathtt{Cl}(w)  
                     &\leq \frac{(n-\frac{k}{2})^2}{6} + \frac{7}{6}(n-\frac{k}{2}) +1 + \sum_{i=\frac{k}{2}+1}^{k}k\\  
                    & = \frac{n^2}{6} +\frac{13k^2}{24} -\frac{nk}{6} +1 + \frac{7}{6}\left(n-\frac{k}{2}\right).
\end{align*}


Now, $\frac{2k+k_1}{k} = \alpha$ implies $\frac{n-k}{k}=\alpha$, i.e., $\frac{k}{n}=\frac{1}{\alpha+1}$. For convenience, we let denote $\beta = \frac{1}{\alpha+1}$.

Then, $$ \frac{\mathtt{Cl}(w)}{n^2} \leq  \frac{1}{6} +\frac{13 \beta^2}{24} -\frac{\beta}{6} +\frac{1}{n^2} + \frac{7}{6}\left(\frac{1}{n}-\frac{\beta}{2n}\right).$$

Thus, 
\begin{align*}
     & \inf \left\{ \frac{\mathtt{Cl}(w)}{n^2} ~: w \in \Sigma^n \cap Fac(v) \text{ is in the form } u^{(1)}uux \text{ with } \frac{2|u|+|u^{(1)}|}{|u|}=\alpha, |u|=|x| \right\}\\
     & \leq  \frac{1}{6} +\frac{13 \beta^2}{24} -\frac{\beta}{6}.
\end{align*} 
This implies by Definition \ref{eq:C_u} that for the closed-rich constant $C_v$ of $v$, we have $C_v\leq \frac{1}{6} +\frac{13 \beta^2}{24} -\frac{\beta}{6}$. 
Let $f(\beta) = \frac{1}{6} + \frac{13}{24} \beta^2 - \frac{1}{6} \beta$. Now, as $2.5< \alpha \leq 3$, we have $\frac{1}{4} \leq \beta < \frac{2}{7}$. On this interval, the function $f(\beta)$ is an increasing function, i.e., $f(\beta) < \frac{8}{49}$ for all $\frac{1}{4} \leq \beta < \frac{2}{7}$.
Thus, $C_v < \frac{8}{49}$.

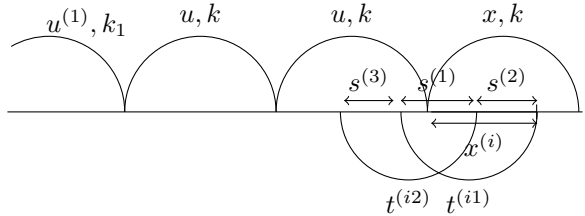
\begin{figure}[h]
\centering
\begin{tikzpicture}
\draw (.55,0) arc[start angle=0, end angle=120, radius=1] node[midway, above] {$u^{(1)}, k_1$};
\draw (2.55,0) arc[start angle=0, end angle=180, radius=1] node[midway, above] {$u, k$};
\draw (4.55,0) arc[start angle=0, end angle=180, radius=1] node[midway, above] {$u, k$};
\draw (6.55,0) arc[start angle=0, end angle=180, radius=1] node[midway, above] {$x, k$};
\draw (6,0) arc[start angle=0, end angle=-180, radius=.9] node[midway, below] {$t^{(i1)}$};
\draw (5.2,0) arc[start angle=0, end angle=-180, radius=.9] node[midway, below] {$t^{(i2)}$};
\draw[-] (-1,0) -- (6.6,0);
\draw[<->] (4.6,-0.15) -- (6,-0.15) node[midway, below] {$x^{(i)}$};
\draw[-|] (4.6,0) -- (6,0);
 \draw[<->] (4.2,0.15) -- (5.16,0.15) node[midway, above] {$s^{(1)}$};
 \draw[<->] (5.2,0.15) -- (6,0.15) node[midway, above] {$s^{(2)}$};
  \draw[<->] (3.45,0.15) -- (4.1,0.15) node[midway, above] {$s^{(3)}$};
\end{tikzpicture}
\caption{Illustration of the proof of Proposition \ref{cont1} for $2.5 < \alpha \leq 3$.} 
\label{yu1}
\end{figure}

\item  \underline{Case $ \alpha > 3$}. Then $v$ always has long cubes. Following the previous case, we have $\beta= \frac{1}{4}$ and  $C_v \leq f(\frac{1}{4}) < \frac{8}{49}$.
\end{itemize}
\end{proof}

\begin{remark}
We remark that the bounds from Propositions \ref{thdf33} and \ref{cont1} can be improved slightly by modifying lengths of $x$ (and also length of prefix for applying 
Theorem \ref{pth} in  Propositions \ref{cont1}) and taking maxmin, as we did in Propositions  \ref{thdf22} and \ref{thdf1}. However, the bound from Proposition \ref{thdf1}, which is the most technical, is the worst one, so imrpoving bounds from other propositions does not affect the global bound; thus we do not do it to reduce technicalities. 
\end{remark}

\begin{proof}[Proof of Theorem \ref{mthsd}]
Since the asymptotic critical exponent of an infinite word over a finite alphabet is always exists and greater than or equal to one (\cite{opovcenska2023asymptotic}), then the proof of the theorem is given by combining Propositions \ref{thdf33}, \ref{thdf331}, \ref{thdf22}, \ref{thdf221},  \ref{thdf1}, \ref{qa11} and \ref{cont1}.
\end{proof}


The following example shows that there exist closed-rich words with arbitrary small closed-rich constants.

\begin{example} \label{exam1201}
 We provide an example showing that for any $\beta> 0$, there always exists an infinite closed-rich word with closed-rich constant smaller than $\beta$. Let $u$ be an infinite closed-rich word over $\Sigma$. Consider a positive integer $m$ such that $\frac{m+1}{m^2}<\beta$ (clearly, such $m$ always exists).
  Let $d \notin \ALPH(u)$ and consider an infinite word $v=d^m u$ over $\Sigma_0 = \Sigma \cup \{ d \}$. The word $v$ is closed-rich, since adding a finite prefix to an infinite word does not affect the closed-rich property. 
     Since $d^m$ has exactly $m+1$ distinct closed factors, by the definition of closed-rich constant, we have $C_v \leq \frac{m+1}{m^2} < \beta$.
\end{example}

Note that in the previous example, factors with small number of closed factors are all of bounded length and in addition occur only a finite number of times in the word. However, there exists uniformly recurrent closed-rich words with arbitrary small closed-rich constants, such that they containing arbitrarily long factor with small number of closed-rich factors:

\begin{example} \label{exam1202}
Corollary 2 from \cite{parshina2024finite} states that Sturmian words with bounded directive sequences of their slopes are closed-rich. Sturmian words with infinitely many numbers greater than $k$ in their directive sequence contain $k$-powers of unbounded period (see, e.g., Chapter 2 in \cite{lothaire}). Now, in the proof of Proposition 5 in \cite{parshina2024finite} the following fact has been explicitly shown: a word of length $n$ and of exponent $k$ contains at most $\frac{n^2}{k}$ distinct closed factors. Choosing $k\geq \frac{1}{\beta}$, one gets the claim.
\end{example}

\section{Closed-rich properties of the Fibonacci word} \label{sec4}

The Fibonacci word $f$ is an infinite closed-rich word as the directive sequence of its slope is bounded \cite{parshina2024finite}.
In this section, we discuss the occurrences of closed factors in any prefix and factor of $f$, and the bounds for its closed-rich constant $C_f$.

\subsection{Numbers of closed factors in prefixes of the Fibonacci word}

In this subsection, we calculate the number of closed factors in any prefix of the Fibonacci word $f$. The main results of the subsection are given by Theorem \ref{Prethecl} giving an explicit formula for the sequence of numbers of distinct closed factors in a prefix of length $n$ of the Fibonacci word, and Theorem \ref{prethe1} providing a formula for its first difference sequence.


Let $(f_n)_{n\geq -1}$ denote the sequence of \textit{Fibonacci words}, where
$f_{-1}=b$, $f_0=a$, and $f_n=f_{n-1}f_{n-2}$ for $n \geq 1$.
The word $f_n$ is called the $n$-th Fibonacci word, and the limit
$\displaystyle f = \lim_{n\rightarrow \infty} f_n$ is the \textit{infinite Fibonacci word}.
We let denote $F_n = |f_n|$; it is straightforward that
the number $F_n$ is the $n$-th Fibonacci number.

In the following lemma we provide some well known properties of the Fibonacci sequence. 

\begin{lemma}\label{malem}
The following properties of the Fibonacci sequence hold true:
    \begin{enumerate} 
        \item  $\displaystyle \sum_{i=-1}^{n-2} {F_i} = F_{n}-1 $.\label{first}
        \item  $\displaystyle \sum_{i=-1}^{n-2} {F_i}^2 = F_{n-2}F_{n-1} $.\label{first2}
        \item   $\displaystyle \lim_{n \to \infty} \frac{F_{n+1}}{F_n} = \phi$
                 where $\phi = \frac{1+\sqrt{5}}{2}$.\label{goldy}
        \item  Finite Fibonacci words are primitive. 
        \item \label{malem:5} If $n \geq 1$ is  even, then $f_n$ ends with $ba$. If  $n \geq 1$  is odd, then $f_n$ ends with $ab$. 
    \end{enumerate}
\end{lemma}

For $n\geq 1$, denote $g_n=f_{n-2}f_{n-1}$. Here is a relation between $f_n$ and $g_n$:
\begin{lemma}\label{pre1}\cite{saari2007periods}  
For $n\geq 2$, we have
$f_n f_{n-1}=f_{n-1} g_n \text{ and } f_{n-1}f_n = f_n g_{n-1}.$
Furthermore, for all $n\geq 1$, the words $f_n$ and $g_n$ differ only by the last two letters.
\end{lemma}

We will make use of the following properties of finite Fibonacci words. We note that some of these properties are true for general standard words (see e.g. \cite[Chapter 2]{lothaire}); however, we provide the proofs for the sake of completeness. Here we denote the prefixes of length $n-1$ and $n-2$ of a word $w$ of length $n$ by $w^{-}$ and $w^{--}$, respectively.

\begin{lemma}\label{lem11}
The following properties of finite Fibonacci words hold true:
    \begin{enumerate}
        \item For $n \geq 2$, $aab$ and $aba$ are the only possible suffixes of length $3$ of $f_n$. \label{impt1}
        \item For $n \geq 0$, any non-empty suffix of $f_{n-1}$ is not a suffix of $f_n$. \label{impt2}
        \item For $n \geq 0 $, $f_{n-1}$ does not occur internally in $f_n$. \label{impt3}
        \item For $n \geq 0 $, $f_{n}$ does not occur internally in $f_{n-1} f_{n}$.\label{impt4}
        
        \item For $n \geq 1$, $f_{n}^{-}$ does not occur internally in $f_{n-1} f_{n}^{-}$ {and $f_n f_n ^{-}$}. \label{impt6}
        
        \item  
        {For $n \geq 3$, $f_{n-1}$ is not a proper suffix of $f_{n}^{-}$.}\label{impt7} 
    \end{enumerate} 
\end{lemma}
\begin{proof}
    \begin{enumerate}
        \item Straightforward.
        \item Follows from Lemma \ref{malem}, Assertion (\ref{malem:5}).  
        \item For $n=0, 1, 2$ the proof is direct. For $n \geq 3$, $f_n f_n = f_{n-1} f_{n-2} f_{n-1} f_{n-2} = f_{n-1} f_{n-1} g_{n-2} f_{n-2}$.  Since $f_{n-1}$ is primitive, $f_{n-1}$ does not occur internally in $f_{n-1} f_{n-1}$. As $f_n$ is a proper prefix of $f_{n-1} f_{n-1}$, we have that $f_{n-1}$ does not occur internally in $f_n$.

        \item For $n=0, 1, 2$ the proof is direct. For $n \geq 3$, $f_{n-1} f_n = f_{n-1} f_{n-1} f_{n-2}$. If $f_n$ occurs internally in $f_{n-1} f_n$, then $f_{n-1}$ must occur internally in $f_{n-1} f_{n-1}$, which is not possible as $f_{n-1}$ is primitive.        


        \item For $n= 1$ the proof is direct.
        Suppose that $f_{n}^{-}$ occurs internally in $f_{n-1} f_{n}^{-}$ for $n \geq 2$. Then for some $u, v \in \Sigma^+$, $f_{n-1} f_{n}^{-} = u f_{n}^{-} v $. This implies $f_{n-1} f_{n}^{-} v^{-1} x = u f_{n} $ for some $x \in \Sigma$. Since $v \in \Sigma^+$, $f_n$ occurs internally in $f_{n-1} f_n$, which is a contradiction by Assertion \ref{impt4}. The proof for the other statement is symmetric.


        
        \item  
        Let  $f_{n-1} \in Suff(f_{n}^{-})$, where $n \geq 3$. Then, $f_n^{-} = u f_{n-1}$ for some $u \in \Sigma^+$. This implies  $f_n = u f_{n-1} x$ for some $x \in \Sigma$, which is a contradiction with Assertion (\ref{impt3}) of this Lemma.
    \end{enumerate}
\end{proof}

\begin{lemma}\label{newlemcombine}
     Let $w$ be a prefix of $f$ with $F_{n-1}+1\leq |w|\leq F_n$, i.e., $w=f_{n-1}\alpha$ for some $\alpha\in Pref (f_{n-2})$ with $1 \leq  |\alpha| \leq F_{n-2}$.
    \begin{itemize}
        \item  If $1 \leq  |\alpha| \leq F_{n-2}-2$, then $f_{n-3} \alpha$ is the largest repeated suffix of $w$ for $n \geq 4$.

        \item If $|\alpha| = F_{n-2}-1$ or $|\alpha| = F_{n-2}$, then $ \alpha$ is the largest repeated suffix of $w$ for $n \geq 3$.



    \end{itemize}
    
\end{lemma} 

\begin{proof}
    \begin{itemize}
        \item Consider $1 \leq  |\alpha| \leq F_{n-2}-2$, where $n \geq 4$ :  Then as $w= f_{n-2} f_{n-3} \alpha$ and $f_{n-2} f_{n-3}^{--} = f_{n-3} f_{n-2}^{--}$, we have $f_{n-3} \alpha \in Pref(f_{n-2} f_{n-3})$. This implies $f_{n-3} \alpha$ is a repeated suffix of $w$. Now, $w= f_{n-3} f_{n-4} f_{n-3} \alpha$. For $u \in \Sigma^+$, let $u f_{n-3} \alpha$ be a repeated suffix of $w$.
    Now, $f_{n-3}$ does not occur internally in $f_{n-4} f_{n-3}$ (by Lemma \ref{lem11}, Assertion (\ref{impt4})). 
   So, $uf_{n-3}\alpha$ has exactly one occurrence in $w$ between the positions $F_{n-3}-|u|+2$ and $|w|$.
   Since $u$ has a non-empty suffix which is also a suffix of $f_{n-4}$, $u$ can not be a suffix of $f_{n-3}$ (by Lemma \ref{lem11}, Assertion (\ref{impt2})). So
   $u f_{n-3} \alpha$  can not occur in $w$ at position $F_{n-3}-|u|+1$. Now, $f_{n-4}$ does not occur internally in $f_{n-3}$ (by Lemma \ref{lem11}, Assertion (\ref{impt3})), $f_{n-4}$ can not be a suffix of $f_{n-3}$ (by Lemma \ref{lem11}, Assertion (\ref{impt2})) and $f_{n-3}$ does not occur internally in $f_{n-3}f_{n-4}$ (by Lemma \ref{lem11}, Assertion (\ref{impt3})). So $uf_{n-3}\alpha$ can not occur in $w$ in any position between $1$ and $F_{n-3}-|u|$. Thus, $u f_{n-3} \alpha$ can not be a repeated suffix of $w$.  Therefore, $f_{n-3} \alpha$ is the longest repeated suffix of $w$. 

        \item  Consider $|\alpha| = F_{n-2}$ or $|\alpha| = F_{n-2}-1$, where $n \geq 3$. Then as $\alpha \in Pref(f_{n-2})$, we have either $\alpha=f_{n-2}$ or $\alpha=f_{n-2} x^{-1}$, where $x \in \Sigma$ is the $F_{n-2}$-th letter of $f$. Combining these two forms of $\alpha$, we write  $\alpha= f_{n-2} y^{-1}$, where $y \in \{\lambda, x\}$.
        Then $w= f_{n-1} f_{n-2} y^{-1}$. For $n=3, 4$, the proof is direct. 
        Consider $n \geq 5$. 
        Using the recurrence formula for finite Fibonacci words, we have $w = f_{n-2} f_{n-3} f_{n-2} y^{-1}$. 
        Clearly, $f_{n-2} y^{-1}$ is a repeated suffix of $w$. For $u \in \Sigma^+$, let $u f_{n-2} y^{-1}$ be a repeated suffix of $w$.
    Now, $f_{n-2} y^{-1}$ does not occur internally in $f_{n-3} f_{n-2} y^{-1}$ (by Lemma \ref{lem11}, Assertions (\ref{impt4}) and (\ref{impt6})). 
   So, $uf_{n-2}y^{-1}$ has exactly one occurrence in $w$ between the positions $F_{n-2}-|u|+2$ and $|w|$.
   Since $u$ has a non-empty suffix which is also a suffix of $f_{n-3}$, $u$ cannot be a suffix of $f_{n-2}$ (by Lemma \ref{lem11}, Assertion (\ref{impt2})). So, 
   $u f_{n-2} y^{-1}$ cannot occur in $w$ at position $F_{n-2}-|u|+1$.
   Now, $f_{n-3}$ does not occur internally in $f_{n-2}$ (by Lemma \ref{lem11}, Assertion (\ref{impt3})); $f_{n-3}$ cannot be a suffix of $f_{n-2} y^{-1}$ (by Lemma \ref{lem11}, Assertions (\ref{impt2}) and (\ref{impt7})), and $f_{n-2} y^{-1}$ does not occur internally in $f_{n-2}f_{n-3}$ (by Lemma \ref{lem11}, Assertions (\ref{impt3}) and (\ref{impt6})).
   So, $uf_{n-2} y^{-1}$ cannot occur in $w$ in positions between $1$ and $F_{n-2}-|u|$. Thus $u f_{n-2} y^{-1}$ cannot be a repeated suffix of $w$. Therefore, $\alpha=f_{n-2} y^{-1}$ is the longest repeated suffix of $w$.
    \end{itemize}
\end{proof}

\begin{lemma}\label{lempref}
    For $n \geq 4$, $f_{n-1}^{--}$ is the largest repeated prefix of $f_n$.
\end{lemma}
\begin{proof}
Using the recurrence formula for finite Fibonacci words, we have 
$f_n = f_{n-2} f_{n-3} f_{n-2}$. 
 For $n \geq 4$, $f_{n-2} f_{n-3}$ and $f_{n-3} f_{n-2}$ differ only by last two letters. Then, $f_{n-2} f_{n-3}^{--} = f_{n-1}^{--}$ is a repeated prefix of $f_n$. 
    Now, $f_n \in Pref(f_{n-1}f_{n-1}^{-})$ and $f_{n-1}^{-}$ does not occur internally in $f_{n-1}f_{n-1}^{-}$ (by Lemma \ref{lem11}, Assertion (\ref{impt6})) implies $f_{n-1}^{-}$ does not occur internally in $f_{n}$. So, $ f_{n-1}^{--}$ is the largest repeated prefix of $f_n$. 
\end{proof}

From Lemma \ref{newlemcombine}, we know that for $n \geq 4$, if $w=f_{n-1} \alpha$,  where $\alpha \in \Sigma^{+} \cap Pref(f_{n-2}^{--})$, then $f_{n-3} \alpha$ is the largest repeated suffix of $w$. To estimate $\mathtt{Cl}(w)-\mathtt{Cl}(w^-)$ using Lemma \ref{lem0}, we need to find the largest repeated suffix of $f_{n-3} \alpha$; the following corollary provides it: 




        

 \begin{corollary}\label{rem1}
    For $n \geq 6$, let $u=f_{n-3} \alpha$, where $\alpha \in 
    Pref(f_{n-2}^{--})$. Then we have the following :
    \begin{itemize}
        \item If $ 1 \leq  |\alpha| \leq F_{n-4} -2$, then $f_{n-5}\alpha$ is the largest repeated suffix of $u$.

        \item If $F_{n-4}-1 \leq |\alpha|  \leq F_{n-2} -2 $, then $\alpha$ is the largest repeated suffix of $u$.


    \end{itemize}
\end{corollary}
\begin{proof}
  We distinguish between the following cases depending on $|\alpha|$:
    \begin{itemize}
        \item If $ 1 \leq  |\alpha| \leq F_{n-4} -2$, then $\alpha$ is a prefix of $f_{n-4}$, since $\alpha \in Pref(f_{n-2}^{--})$.
        So, by Lemma \ref{newlemcombine}, $f_{n-5} \alpha$ is the largest repeated suffix of $u$.

        \item  If  $F_{n-4}-1 \leq |\alpha| \leq F_{n-4}$, then  $\alpha$ is a prefix of $f_{n-4}$, since $\alpha \in Pref(f_{n-2}^{--})$. So, by Lemma \ref{newlemcombine}, $\alpha$ is the largest repeated suffix of $u$. 



        \item If $F_{n-4}+1 \leq |\alpha| \leq F_{n-2} -2$, then as $\alpha \in Pref(f_{n-2}^{--})$, i.e.,  $\alpha \in Pref(f_{n-4} f_{n-3}^{--})$, we have $\alpha = f_{n-4} \gamma$ for some non-empty $\gamma \in Pref(f_{n-3}^{--})$. 
        This implies $u= f_{n-3} f_{n-4} \gamma = f_{n-2} \gamma$. 
Then by Lemma \ref{newlemcombine}, $f_{n-4} \gamma$ is the largest repeated suffix of $u$.

    \end{itemize}
\end{proof}

We let $\mathtt{PCl}(n)$ denote number of distinct closed factors in the prefix of $f$ of length $n$, and we let $s(n)= \mathtt{PCl}(n)-\mathtt{PCl}(n-1)$ denote the first difference sequence of $\mathtt{PCl}(n)$. 
Theorem \ref{prethe1} provides the structure of the sequence $\{s(n)\}_{n \geq 1}$:

\begin{theorem}\label{prethe1}
    The first difference sequence of the sequence $\mathtt{PCl}(n)$  of closed factors in the prefix of length $n$ of the Fibonacci word is given by the following:
    \begin{equation}\label{eq:s_pref}\{s(n)\}_{n \geq 1} = \underbrace{F_{-1}, F_{-1}}_{2F_{-1} \text{terms}},  \underbrace{F_0, F_0,}_{2F_0 \text{terms}} \underbrace{F_1, \dots, F_1}_{2F_1 \text{terms}},   \dots, \underbrace{F_m, F_m, \dots, F_m,}_{2F_m \text{ terms}} \dots\end{equation}
\end{theorem}

In other words, if we denote repetitions in a sequence by exponents and concatenation by the symbol $\displaystyle \prod$, then we can rephrase the theorem as follows:

   $$\{s(n)\}_{n \geq 1}= \displaystyle \prod_{i=-1}^{\infty} {F_i}^{2F_i}.$$

\begin{proof}
    We first show that for each $n \geq 1$ the element $s(n)$ is a Fibonacci number. Then we prove that for each Fibonacci number $F_i$, $i \geq -1$, there are exactly  $2F_i$ consecutive terms, each equal to $F_i$, in the sequence $\{s(n)\}_{n\geq 1}$, and such block is immediately followed by a block of $2F_{i+1}$ consecutive terms, each equal to $F_{i+1}$.
    
    For $n\in \mathbb{N}$, we denote $w =Pref_{n-1}(f)$, $x\in\Sigma$ the $n$-th letter of $f$, so that $wx =Pref_n (f)$.
    For $n \leq F_6-1$, we directly calculate  $s(n)$, where $s(1)=s(2)=\cdots =s(4)=1$, $s(5)=s(6)=\cdots=s(8)=2, s(9)=s(10)=\cdots=s(14)=3, s(15)=s(16)=\cdots=s(20)=5$.
    Consider $n \geq F_6$. If $t$ is the largest repeated suffix of $wx$ and $z$ is the largest repeated suffix of $t$, then by Lemma \ref{lem0},  $s(n) = |t|-|z|$. 
    Since for each $m \geq 0$, $f_m$ is a prefix of $f$, we have the following cases: 
    \begin{itemize}

        \item \textbf{Case I:} Suppose that $ n \in \{ F_i, F_i-1\}$ for some $i \geq 6$. Then either $wx=f_i$, or $wx = f_i^{-}$. By Lemma \ref{newlemcombine}, for $wx=f_i$  we get $|t|-|z| = F_{i-2}-F_{i-4} = F_{i-3}$ and for $wx=f_i^{-}$ we get $|t|-|z| = (F_{i-2}-1)-(F_{i-4}-1) = F_{i-3}$.

        
        
        \item \textbf{Case II:} Suppose that $F_{i-1}+1 \leq n \leq F_{i}-2$ for some $i\geq 7$. 
        Then $wx=f_{i-1}\alpha$, where $\alpha \in \Sigma^{+} \cap Pref(f_{i-2}^{--})$. 
        Then by Lemma \ref{newlemcombine}, $t= f_{i-3}\alpha$. Using Corollary \ref{rem1}, we now find $z$:
        \begin{itemize}

            \item If $1 \leq |\alpha| \leq F_{i-4}-2$, then $z=f_{i-5}\alpha$. This implies $|t|-|z|=F_{i-4}$.

            \item If $ F_{i-4}-1 \leq |\alpha| \leq F_{i-2}-2$, then $z= \alpha$. This implies $|t|-|z| = F_{i-3}$.


        \end{itemize}
    \end{itemize}
    Therefore, each element of the sequence $\{s(n)\}_{n\geq 1}$ is a Fibonacci number.

     We directly calculated  $s(1)=\cdots =s(4)=1$, $s(5)=\cdots=s(8)=2, s(9)=\cdots=s(14)=3, s(15)=\cdots=s(24)=5$.
    Let $j \geq 7$. Now, from the above Cases I and II, we see that $s(n)=F_{i-3}$ for $n \in [F_{i-1}+F_{i-4}-1,F_i+F_{i-3}-2]$. So, we have $(F_i+F_{i-3}-2)-(F_{i-1}+F_{i-4}-1)+1=F_{i-2}+F_{i-5}=2F_{i-3}$ consecutive elements $F_{i-3}$ in the sequence $\{s(n)\}_{n\geq 1}$. Furthermore, any block of $2F_i$ consecutive terms, each equal to $F_i$, is immediately followed by a block of $2F_{i+1}$ consecutive terms, each equal to $F_{i+1}$.    
\end{proof}

\begin{lemma}\label{secfir}
    For $n \geq 3$, the following holds: $$ 2(F_{-1} + F_0+F_1+\cdots + F_{n-4}) + F_{n-3} +2 = F_n.$$ 
\end{lemma}
\begin{proof}
    From Lemma \ref{malem}, Assertion (\ref{first}), we know that  $F_{-1} + F_0+F_1+\cdots+F_{n-2} = F_n-1$ for $n \geq 1$. Then, for $n \geq 3$, we have
    $ 2(F_{-1} + F_0+F_1+\cdots + F_{n-4}) + F_{n-3} +2 = 2(F_{n-2}-1) + F_{n-3} +2 = 2 F_{n-2} +F_{n-3} = F_{n-2} + F_{n-1} = F_n$.
\end{proof}

Using Theorem \ref{prethe1} and Lemma \ref{secfir}, we now 
count the number of closed factors in any prefix of $f$.

\begin{theorem}\label{Prethecl}
  For $n \geq 4$, let $p$ be a prefix of $f$ of length $l$, where $l=F_n+k$ and $0 \leq k < F_{n-1}$.  Then 
  $$\mathtt{Cl}(p) = \begin{cases} 1 + F_{n-3} (F_{n-2} + F_{n-4} +k+2), &\mbox{ if }  0 \leq  k \leq F_{n-3}-2, \\ 
   1 + F_{n-2} (F_{n-3} + k +2), &\mbox{ if }  F_{n-3} -2 < k \leq F_{n-1}-1;
  \end{cases}$$
and the initial values are $\mathtt{PCl}(1)=2$, $\mathtt{PCl}(2)=3$, $\mathtt{PCl}(3)=4$, $\mathtt{PCl}(4)=5$, $\mathtt{PCl}(5)=7$, $\mathtt{PCl}(6)=9$, $\mathtt{PCl}(7)=11$. 
\end{theorem}

\begin{remark}
    The statement of this theorem
    can be reformulated using Zeckendorf representation. Namely, every non-negative integer can be represented as a sum of Fibonacci numbers $\{F_{i}\}_{i\geq 0}$ with no two consecutive Fibonacci numbers, and such representation is unique \cite{zeckendorf1972representations}. So, the condition $F_n \leq l < F_{n+1}$ on the length $l$ can be reformulated as follows: in Zechendorf representaton, $l$ has the most significant digit at position $n+1$ (corresponding to $F_n$). 
\end{remark}

\begin{proof}
     We can count the number of distinct closed factors in a prefix $p$ of $f$ of length $l$ as follows:
     \begin{equation}\label{eq231}
        \mathtt{Cl}(p) = \mathtt{PCl}(0)+ \sum_{j=1}^{l} (\mathtt{PCl}(j)-\mathtt{PCl}(j-1))  = 1 + \sum_{j=1}^{l} s(j).
    \end{equation}
     Now using \eqref{eq:s_pref} for the sequence $\{s(m)\}_{m \geq 1}$ from Theorem  \ref{prethe1},
     to compute  $\mathtt{Cl}(p)$, we need to count the sum of the first $l$ terms of $\{s(m)\}_{m \geq 1} $.


     \begin{itemize}


         \item Suppose  that $0 \leq k \leq F_{n-3}-2$. Then $2+k \leq F_{n-3}$, i.e.,  $F_{n-3}+2+k \leq 2F_{n-3}$. Now, by Lemma \ref{secfir}, $F_n+k= 2(F_{-1} + F_0+F_1+\cdots + F_{n-4}) + F_{n-3} +2 +k$. Then as $F_{n-3}+2+k \leq 2F_{n-3}$, the sum of the first $F_n+k$ terms of the sequence $\{s(m)\}_{m \geq 1} $ is given by the following:
         \begin{align*}
            \sum_{j=1}^{F_n+k} s(j) &= \underbrace{(F_{-1} + F_{-1})}_{2F_{-1} \text{terms}} +  \underbrace{(F_0 + F_0)}_{2F_0 \text{ terms}} +   \dots + \underbrace{(F_{n-4} + F_{n-4} + \dots + F_{n-4})}_{2F_{n-4} \text{ terms}} \\
            & + \underbrace{(F_{n-3} + F_{n-3} + \dots + F_{n-3})}_{{F_{n-3} +2+k} \text{ terms}}\\  
            &=\sum_{i=-1}^{n-4} 2 {F_i}^2 + ({F_{n-3} +2+k})F_{n-3}\\
            &= 2 F_{n-4} F_{n-3}+ ({F_{n-3} +2+k})F_{n-3} ~\hspace{0.2cm} (\text{Using Lemma }\ref{malem}, \text{Assertion} (\ref{first2}))\\
            &=F_{n-3} (F_{n-2} + F_{n-4} +2+k).
        \end{align*} 
          Thus from (\ref{eq231}) we have
         $ \mathtt{Cl}(p) =  1+ F_{n-3} (F_{n-2} + F_{n-4} +2+k)$.


         \item Suppose that $F_{n-3} -2 < k \leq F_{n-1}-1 $. Then $k= F_{n-3}-2 + k_1$, where $1 \leq k_1 \leq F_{n-2}+1$.
         Now, by Lemma \ref{secfir}, $l=F_n+k= 2(F_{-1} + F_0+F_1+\cdots + F_{n-4}) + F_{n-3} +2 + F_{n-3}-2 + k_1 = 2(F_{-1} + F_0+F_1+\cdots + F_{n-3}) + k_1$.    
        Then we have
         \begin{align*}
            \sum_{j=1}^{l} s(j) &= \underbrace{(F_{-1} + F_{-1})}_{2F_{-1} \text{terms}} +  \underbrace{(F_0 + F_0)}_{2F_0 \text{terms}} +   \dots + \underbrace{(F_{n-3} + F_{n-3} + \dots + F_{n-3})}_{2F_{n-3} \text{ terms}} \\
            &+ \underbrace{(F_{n-2} + F_{n-2} + \dots + F_{n-2})}_{{k_1} \text{ terms}}\\  
            &=\sum_{i=-1}^{n-3} 2 {F_i}^2 + k_1 F_{n-2}\\
            &= 2 F_{n-3} F_{n-2} +k_1 F_{n-2} ~\hspace{1cm} (\text{Using Lemma }\ref{malem}, \text{Assertion} (\ref{first2}))\\
            &= 2 F_{n-3} F_{n-2}+ (k-F_{n-3}+2) F_{n-2}\\
            &= F_{n-2} (F_{n-3} + k +2).
        \end{align*} 
        Thus from (\ref{eq231}) we have
                $\mathtt{Cl}(p) =1+  F_{n-2} (F_{n-3} + k +2)$.      


     \end{itemize}
 \end{proof}

\begin{remark} Theorem \ref{Prethecl} can be verified by Walnut \cite{mousavi2016automatic} using the following commands:
\begin{enumerate}
    \item def factorequal ``?msd\_fib At~ t$<$n $=>$ F[i+t]=F[j+t]":      
    
    \noindent This creates a predicate \$factorequal(i, j, n) that is true if and only if $f[i \hdots i+n-1]$ and $f[j \hdots j+n-1]$ are equal.

    \item def occurfac``?msd\_fib m$<=$n \& Ek k+m$<=$n \& \$factorequal(i,j+k,m)":

     \noindent This creates a predicate \$occurfac(i,j,m,n) that is true if and only if  $f[i \hdots i+m-1]$ is a factor of $f[j \hdots j+n-1]$.

     \item def border ``?msd\_fib m$>=$1 \& m$<$n \& \$factorequal(i,i+n-m,m)":

      \noindent This creates a predicate \$border(i,m,n) that is true if and only if  $f[i \hdots i+m-1]$ is a border of $f[i \hdots i+n-1]$.

    \item def closedword ``?msd\_fib (n$<=$1) $|$ Ej j$<$n \& \$border(i,j,n) \& $\sim$\$occurfac(i,i+1,j,n-2)":

  \noindent This creates a predicate \$closedword(i,n) that is true if and only if  $f[i \hdots i+n-1]$ is a closed word.

     \item def insideclosedfactor ``?msd\_fib m$>=$1 \& \$closedword(i,m) \& \$occurfac(i,j,m,n)":

 \noindent This creates a predicate \$insideclosedfactor(i,j,m,n) that is true if and only if  $f[i \hdots i+m-1]$ is a non-empty closed factor of $f[j \hdots j+n-1]$.

     \item def novelclosedfactorfib ``?msd\_fib \$insideclosedfactor(i,j,m,n) \& Ak (k$>=$j \& (k+m)$<=$(j+n)) \& \$factorequal(i,k,m) $=>$ k$>=$i":

 \noindent This creates a predicate \$novelclosedfactorfib(i,j,m,n) that is true if and only if  $f[i \hdots i+m-1]$ is a non-empty closed factor of $f[j \hdots j+n-1]$ which occurs exactly once in $f[j \hdots j+n-1]$.

     \item def prefixnovelclosedfactorfib n ``?msd\_fib \$novelclosedfactorfib(i,0,m,n)":

 \noindent This creates a predicate \$prefixnovelclosedfactorfib(i,m,n) that is true if and only if  $f[i \hdots i+m-1]$ is a non-empty closed factor of $f[0 \hdots n-1]$ which occurs exactly once in $f[0 \hdots n-1]$.
\end{enumerate}

This gives a linear representation for $\mathtt{Cl}(p)$, from which we can verify Theorem \ref{Prethecl}.

\end{remark}

\subsection{Minimum Number of closed factors in any factor of $f$}\label{sec5}

In this subsection, we estimate the number of distinct closed factors in any factor of the Fibonacci word $f$. We begin with determining the number of closed factors in any factor of $f$ of length $F_n$. Using this value, we then estimate the minimum number of closed factors in any factor of $f$. 


First we observe following. For any $n \geq 2$, $f_n^2$ is a prefix of $f$. Then as $f_n$ is primitive,  all distinct $F_n$ conjugates of $f_n$ occur in $f_n^2$ for $n \geq 2$.
So, for $n \geq 2$, the set of conjugates of $f_n$ is $$C(f_n) = \{ u^{-1} f_n u \mid u \in Pref(f_n), 0 \leq |u| \leq F_n-1\}.$$
Now, $f$ has $F_n+1$ distinct factors of length $F_n$. Then there exists a factor of $f$ of length $F_n$ which does not lie in $C(f_n)$. This word is called the \textit{singular} word of length $F_n$; we denote it by $w_n$. We remark that this is true in general for every standard factor of any Sturmian word (see, e.g., Chapter 2 is \cite{lothaire}). For the Fibonacci word, since its directive sequence is $(ab)^{\infty}$, we also have that when $n$ is odd (resp. even), then $w_n = a p_n a$ (resp, $w_n = b p_n b$), where  $p_n$ denotes the palindromic prefix of $f_n$ of length $F_n-2$ (see also \cite{ficiu}).

 The following proposition gives the difference between the length of the largest repeated suffix $s$ of a factor of the Fibonacci word of length $F_n$ and the length of the largest repeated suffix of $s$ (and a symmetric statement with prefixes). We recall that this expression is needed to estimate the number of closed factors using Lemma \ref{lem0}.


\begin{proposition}\label{prop1}
    For $n \geq 7 $, let $z$ be a factor of $f$ of length $F_n$,  $s$ be the largest repeated suffix of $z$, $r$ be the largest repeated suffix of $s$, $p$ be the largest repeated prefix of $z$ and $q$ be the largest repeated prefix of $p$. Then the following hold true:
    \begin{enumerate}


        \item Let $z=u^{-1} f_n u$, where $ u \in Pref(f_{n-3}^{--})$. 
        Then  $|s|-|r| = |p|-|q| = F_{n-3}$. 
    
        \item Let $z = u^{-1} f_n u$, where $u=f_{n-3}^{-}$. Then $|s|-|r|=F_{n-4}+1$ and $|p|-|q| = F_{n-3} + 1$.     


        \item Let $z=u^{-1} f_n u$, where $F_{n-3} \leq |u|\leq F_{n-2} -1$, i.e. $u=f_{n-3} \alpha$, $\alpha \in Pref(f_{n-4}^{-})$.
            \begin{enumerate}
                \item If $0 \leq |\alpha| \leq F_{n-4}-2$, then  $|s|-|r|=|p|-|q| = F_{n-4}$.
                \item If $ |\alpha| = F_{n-4}-1$, then  $|s|-|r| = F_{n-3}+1$ and $|p|-|q|=F_{n-4}+1$.  
            \end{enumerate}



        \item Let $z=u^{-1} f_n u$, where $F_{n-2} \leq |u|\leq F_{n-1} - 2$, i.e. $u=f_{n-2} \alpha$, $\alpha \in Pref(f_{n-3}^{--})$. 
        Then $|s|-|r|= |p|-|q| = F_{n-3}$.      

        \item Let $z=u^{-1} f_n u$, where $u=f_{n-1}^{-}$. Then $|s|-|r| = F_{n-2}$ and $|p|-|q| = F_{n-3}$.
        
        \item Let $z=u^{-1} f_n u$, where $u=f_{n-1}$. Then,  $|s|-|r| = F_{n-2}$ and $|p|-|q| = F_{n-3}+2$.
    
        \item Let $z=u^{-1} f_n u$, where $F_{n-1}+1 \leq |u|\leq F_{n} -2$, i.e. $u=f_{n-1} \alpha$, $\alpha \in Pref(f_{n-2})$ and $ 1 \leq |\alpha| \leq F_{n-2}-2 $. 
        \begin{enumerate}
             \item If $ 1 \leq |\alpha| \leq F_{n-4}-2 $, then, $|s|-|r| = F_{n-2} - |\alpha|$ and $|p|-|q| = F_{n-3} + |\alpha| +2$.




                 \item  If $ F_{n-4} -1 \leq |\alpha| \leq F_{n-3}-1 $, then $|s|-|r|=|p|-|q| = F_{n-2} $.

                \item If $|\alpha| = F_{n-3}$, then $|s|-|r|= F_{n-2} $ and $ |p|-|q| = F_{n-3} +2$.  
        
            \item If $ F_{n-3} +1 \leq |\alpha| \leq F_{n-2}-2 $, then  $|s|-|r| = F_{n-1} - |\alpha|$ and $|p|-|q| = |\alpha| + 2$.
            
        \end{enumerate}

        \item Let $z=u^{-1} f_n u$, where $u=f_{n}^{-}$. Then,  $|s|-|r| = F_{n-3}$ and $|p|-|q| = F_{n-2}$.

        \item Let $z=w_n$. Then $|s|-|r|=|p|-|q| = F_{n-3}$.
    \end{enumerate}
\end{proposition}
\begin{proof}

    We provide a proof of Assertion (1). Assertions (2)--(9) are proved in a similar way using  Lemma \ref{lem11}.
    


    
    For the proof of Assertion (1), consider $ u \in Pref(f_{n-3}^{--})$. Then $f_{n-3}=uv$ for some $v \in \Sigma^+$.
         Now using the recurrence formula for finite Fibonacci words,   $z=u^{-1} f_n u = u^{-1} f_{n-3} f_{n-4} f_{n-3} f_{n-2} u = v f_{n-4}  f_{n-3} f_{n-3} f_{n-4} u $. 
        As $0\leq |u| \leq F_{n-3}-2$, $f_{n-3} f_{n-4} u \in Pref(f_{n-3} f_{n-3} f_{n-4})$. 
        So, $f_{n-3} f_{n-4} u $ is a repeated suffix of $z$.
        For some $ \alpha \in \Sigma^+$, consider $\alpha f_{n-3} f_{n-4} u $ which is a repeated suffix of $z$.
        As $f_{n-3}$ is a primitive word, $f_{n-3}$ does not occur internally in $f_{n-3}f_{n-3}$. So, $\alpha f_{n-3} f_{n-4} u $ has exactly one occurrence in $z$ between the positions $|v|+F_{n-4}-|\alpha| +2$ and $|z|$.
        Since $\alpha$ has a non-empty suffix which is also a suffix of $f_{n-3}$, $\alpha$ cannot be a suffix of $vf_{n-4}$  (by Lemma \ref{lem11}, Assertion (\ref{impt2})). So, $\alpha f_{n-3} f_{n-4} u $ cannot occur in $z$ at position $|v|+F_{n-4}-|\alpha| +1$. 
        Now, $f_{n-3}$ does not occur internally in $f_{n-4}f_{n-3}$  (by Lemma \ref{lem11}, Assertion (\ref{impt4})). So, $\alpha f_{n-3} f_{n-4} u $ has exactly one occurrence in $z$ between the positions $|v| -|\alpha| +2$ and $|z|$. Since  $f_{n-3}f_{n-4} \neq f_{n-4}f_{n-3}$, the word $\alpha f_{n-3} f_{n-4} u $ cannot occur in $z$ at position $|v| -|\alpha| +1$.
        Now, $f_{n-4}$ does not occur internally in $f_{n-3}$ (by Lemma \ref{lem11}, Assertion (\ref{impt3})), $f_{n-4}$ can not be a suffix of $f_{n-3}$  (by Lemma \ref{lem11}, Assertion (\ref{impt2})) and $f_{n-3}$ does not occur internally in $f_{n-2}$ (by Lemma \ref{lem11}, Assertion (\ref{impt3})). 
        Then as $\alpha \in \Sigma^+$, the word $\alpha f_{n-3} f_{n-4} u $ can not occur in $z$ between the positions $1$ and $|v| -|\alpha| $.
        Thus, $s = f_{n-3} f_{n-4} u =f_{n-2}u$.
        Then by Lemma \ref{newlemcombine},  $r = f_{n-4}u$. 
        Therefore, $|s|-|r|=F_{n-3}$.

        Now, $z = v f_{n-4}  f_{n-3} f_{n-3} f_{n-4} u $. Since $v \in Suff(f_{n-3})$, $ |u| \leq F_{n-3}-2$, and $f_{n-4} f_{n-3}^{--} \in Pref(f_{n-3} f_{n-4})$, we have   $  v f_{n-4} f_{n-3}^{--}$ is a repeated prefix of $z$. If possible, let $v f_{n-4} f_{n-3}^{-}$ be a repeated prefix of $z$.  
        Now, $f_{n-3}^{-}$ does not occur internally in $f_{n-3}f_{n-3}^{-}$ (by Lemma \ref{lem11}, Assertion (\ref{impt6})), $f_{n-4}$ can not be a suffix of $f_{n-3}$ (by Lemma \ref{lem11}, Assertion (\ref{impt2})), 
         $f_{n-3}^{-}$ does not occur internally in $f_{n-2}$ (by Lemma \ref{lem11}, Assertion (\ref{impt6})), 
        $f_{n-4}$ can not be a proper suffix of $f_{n-3}^{-}$ (by Lemma \ref{lem11}, Assertion (\ref{impt7})), and $f_{n-4}$ does not occur internally in $f_{n-3}$ (by Lemma \ref{lem11}, Assertion (\ref{impt3})).
        Thus, $v f_{n-4} f_{n-3}^{-}$ can not be a repeated prefix of $z$. Therefore, 
        $p = v f_{n-4} f_{n-3}^{--}$.
        Now, $p=v f_{n-4} f_{n-3}^{--} = v f_{n-3} f_{n-4}^{--}$. 
        Clearly, $v f_{n-4} ^{--}$ is a repeated prefix of $p$. 
        Using similar process of finding $p$, we can show that $q=v f_{n-4}^{--}$. So, $|p|-|q|= F_{n-3}$.
\end{proof}



We now examine the difference between the numbers of distinct closed factors in two consecutive factors of length $F_n$ of $f$. To do that, we need the following observations.
Consider $\alpha = xu$ and $\beta = u y$, where $x , y \in \Sigma$, $u \in \Sigma^+$ and  $\ALPH(u)=\ALPH(\alpha) = \ALPH(\beta)$. Let $p$ be the largest repeated prefix of $\alpha$, $q$ be the largest repeated prefix of $p$,  $s$ be the largest repeated suffix of $\beta$ and $r$ be the largest repeated suffix of $s$. 
Then by Lemmas \ref{lem0}  and \ref{lem01}, we have $\mathtt{Cl}(\alpha) - \mathtt{Cl}(u) = |p|-|q|$ and  $\mathtt{Cl}(\beta) - \mathtt{Cl}(u) = |s|-|r|$.
This implies 
\begin{equation}\label{equ1}
    \mathtt{Cl}(\beta) - \mathtt{Cl}(\alpha) = (|s|-|r|) - (|p|-|q|).
\end{equation}
For $w=xy$ and $|x|=i$, let $\sigma^i(w) = yx$. 
Then $C(w)=\{\sigma^i(w) ~ | ~ 0 \leq i \leq |w|-1\}$.
We know that each of the $F_n$ conjugates of $f_n$ is a factor of $f_n^2$ for $n \geq 2$. Then by \eqref{equ1}, for $1 \leq i \leq F_n-1$, we have
\begin{equation}\label{equ2}
    \mathtt{Cl}(\sigma^i(f_n)) - \mathtt{Cl}(\sigma^{i-1}(f_n)) = (|s|-|r|) - (|p|-|q|),
\end{equation}
where $s$ is the largest repeated suffix of $\sigma^i(f_n)$, $r$ is the largest repeated suffix of $s$, $p$ is the largest repeated prefix of $\sigma^{i-1}(f_n)$ and $q$ is the largest repeated prefix of $p$.
Furthermore, Chuan and  Ho (Theorem 2.8; \cite{chuan2005locating}) discussed the position of the first occurrence of $w_n$ in $f$.
\begin{theorem}\label{imp1}\cite{chuan2005locating}
    For $n \geq 2$, the first occurrence of $w_n$ in $f$ is at the position $F_{n+1}-1$. In addition, if $v \in \Sigma^{F_n}$ occurs at the position $F_{n+1}-2$ in $f$, then $v=xy f_{n-2} u$, where $f_{n-1}=uxy$, $x, y \in \Sigma$.  
\end{theorem}
Thus by Assertion $4$ of Proposition \ref{prop1}, equation \eqref{equ1} and Theorem \ref{imp1} we have
\begin{equation}\label{equ3}
    \mathtt{Cl}(w_n) - \mathtt{Cl}(v) = (|s|-|r|) - F_{n-3},
\end{equation}
where $s$ is the largest repeated suffix of $w_n$, $r$ is the largest repeated suffix of $s$ and $v \in \Sigma^{F_n} \cap C(f_n)$ occurs at $F_{n+1}-2$ in $f$.
Thus, from Proposition \ref{prop1} and Equations (\ref{equ2}), $(\ref{equ3})$, we have the following observation.
\begin{observation}\label{prop4}
    Let $n \geq 7$ and $w=f_n$. Then, we have the following:
    \begin{enumerate}

        \item If $1 \leq  i \leq F_{n-3}-2$, then $\mathtt{Cl}(\sigma^i(w)) - \mathtt{Cl}(\sigma^{i-1}(w)) = F_{n-3} - F_{n-3} = 0$.

        \item If $i=F_{n-3}-1$, then $\mathtt{Cl}(\sigma^i(w)) - \mathtt{Cl}(\sigma^{i-1}(w)) = F_{n-4}+1 - F_{n-3} = 1 - F_{n-5}$.

        \item If $i=F_{n-3}$, then $\mathtt{Cl}(\sigma^i(w)) - \mathtt{Cl}(\sigma^{i-1}(w)) = F_{n-4} - F_{n-3} -1 = -F_{n-5} - 1$.



         \item If $F_{n-3}+1 \leq  i \leq F_{n-2}-2$, then $\mathtt{Cl}(\sigma^i(w)) - \mathtt{Cl}(\sigma^{i-1}(w)) = F_{n-4} - F_{n-4} = 0$.

        \item If $i=F_{n-2}-1$, then $\mathtt{Cl}(\sigma^i(w)) - \mathtt{Cl}(\sigma^{i-1}(w)) = F_{n-3} +1 - F_{n-4} = 1+F_{n-5}$.

        \item If $i= F_{n-2}$, then $\mathtt{Cl}(\sigma^i(w)) - \mathtt{Cl}(\sigma^{i-1}(w)) = F_{n-3} - F_{n-4} -1 = F_{n-5} - 1$.



        \item If $F_{n-2}+1 \leq i \leq F_{n-1} - 2$, then  $\mathtt{Cl}(\sigma^i(w)) - \mathtt{Cl}(\sigma^{i-1}(w)) = F_{n-3} - F_{n-3} = 0$.



         \item If $F_{n-1} - 1 \leq i \leq F_{n-1}$, then $\mathtt{Cl}(\sigma^i(w)) - \mathtt{Cl}(\sigma^{i-1}(w)) = F_{n-2} - F_{n-3} = F_{n-4}$.

        \item If $i=F_{n-1} +1$, then $\mathtt{Cl}(\sigma^i(w)) - \mathtt{Cl}(\sigma^{i-1}(w)) = F_{n-2} - 1 - F_{n-3} - 2 = F_{n-4} - 3$.

        \item If $i=F_{n-1} + |\alpha|$, where $2 \leq |\alpha| \leq F_{n-4}-2$, then 
        $\mathtt{Cl}(\sigma^i(w)) - \mathtt{Cl}(\sigma^{i-1}(w)) = F_{n-2} - |\alpha| - |\alpha| -1 - F_{n-3} = F_{n-4} - 2 |\alpha| -1$.





        \item If $F_{n-1} + F_{n-4} -1 \leq i \leq F_{n-1} + F_{n-3} $, then $\mathtt{Cl}(\sigma^i(w)) - \mathtt{Cl}(\sigma^{i-1}(w)) = F_{n-2} - F_{n-2} = 0$.

        \item If $i = F_{n-1} + F_{n-3} +1 $, then $\mathtt{Cl}(\sigma^i(w)) - \mathtt{Cl}(\sigma^{i-1}(w)) = (F_{n-1} - F_{n-3} -1) - (F_{n-3} + 2) = F_{n-4} -3$.

        \item If $i=F_{n-1} + \gamma $, where $F_{n-3} + 2 \leq |\gamma| \leq F_{n-2}-2$,
        then $\mathtt{Cl}(\sigma^i(w)) - \mathtt{Cl}(\sigma^{i-1}(w)) = F_{n-1} - |\gamma| - |\gamma| - 1 = F_{n-1} - 2 |\gamma| - 1 $.

        \item If $i = F_n -1$, then $\mathtt{Cl}(\sigma^i(w)) - \mathtt{Cl}(\sigma^{i-1}(w)) = F_{n-3} - F_{n-2} = -F_{n-4}$.

        \item If $ i = F_n$, then $\mathtt{Cl}(\sigma^i(w)) - \mathtt{Cl}(\sigma^{i-1}(w)) = F_{n-3} - F_{n-2} = -F_{n-4}$.

        \item Let $v \in \Sigma^{F_n}$ occur at $F_{n+1}-2$ in $f$. Then, by Equation \ref{equ3}, $\mathtt{Cl}(w_n) - \mathtt{Cl}(v) = F_{n-3} - F_{n-3}=0$.
        
    \end{enumerate}
\end{observation}
Using Observation \ref{prop4}, we now calculate number of closed factors in any factor of $f$ of length $F_n$, $n \geq 7$. 
\begin{theorem}\label{nth}
     For $n \geq 7$, let $w=f_n$ and $\mathtt{Cl}(f_n) = m$. Then,
      \begin{enumerate}
        
        
        \item For $0 \leq  i \leq F_{n-3}-2$,   $\mathtt{Cl}(\sigma^i(w)) = m$.

        \item For $i=F_{n-3}-1$, $\mathtt{Cl}(\sigma^i(w)) = m + 1 - F_{n-5}$.



        \item For $F_{n-3} \leq  i \leq F_{n-2}-2$, $\mathtt{Cl}(\sigma^i(w)) = m - 2F_{n-5}$.

        \item For $i=F_{n-2}-1$, $\mathtt{Cl}(\sigma^i(w)) = m - 2F_{n-5} + 1+F_{n-5} =  m - F_{n-5} + 1$.



        \item For $F_{n-2} \leq  i \leq F_{n-1} - 2$, $\mathtt{Cl}(\sigma^i(w)) = m$.

        \item For $i=F_{n-1} - 1$, $\mathtt{Cl}(\sigma^i(w)) = m+ F_{n-4}$.

        \item For $i=F_{n-1}$, $\mathtt{Cl}(\sigma^i(w)) = m + 2F_{n-4}$.

        \item For $i=F_{n-1} +1$, $\mathtt{Cl}(\sigma^i(w)) = m + 3 F_{n-4} - 3$.

        \item For $i=F_{n-1} + |\alpha| $, where $1 < |\alpha| \leq F_{n-4}-2$, $\mathtt{Cl}(\sigma^i(w)) =  m + 3 F_{n-4} - 3 + F_{n-4} - 2 |\alpha| -1 = m + 4 F_{n-4} - 2 |\alpha| -4 $.





        \item For $F_{n-1} + F_{n-4} -1 \leq i \leq F_{n-1} + F_{n-3} $, $\mathtt{Cl}(\sigma^i(w)) = m + 2 F_{n-4}$.

        \item For $i = F_{n-1} + F_{n-3} +1 $, $\mathtt{Cl}(\sigma^i(w)) = m + 2 F_{n-4} + F_{n-4} -3 = m + 3 F_{n-4} -3$.

        \item For $i=F_{n-1} + |\gamma| $, where $F_{n-3} + 1 < |\gamma| \leq F_{n-2}-2$, $\mathtt{Cl}(\sigma^i(w)) =  m + 3 F_{n-4} -3 + F_{n-1} - 2 |\gamma| - 1 = m + 3 F_{n-4} + F_{n-1} - 2 |\gamma| - 4$.


        \item For $i = F_n -1$, $\mathtt{Cl}(\sigma^i(w)) = m + 3 F_{n-4} + F_{n-1} - 2 F_{n-2} +4 - 4 -F_{n-4} = m +2 F_{n-4} + F_{n-1} - 2 F_{n-2} $.

        \item For $ i = F_n$, $\mathtt{Cl}(\sigma^i(w)) = m +2 F_{n-4} + F_{n-1} - 2 F_{n-2} -F_{n-4} = m + F_{n-4} + F_{n-1} - 2 F_{n-2} = m + F_{n-4} + F_{n-3} - F_{n-2} = m $.

        \item Let $w_n$ be the singular word. Then $\mathtt{Cl}(w_n) =  \mathtt{Cl}(\sigma^{F_{n-1}-2}(w))= m $.        
     
    \end{enumerate}
     
\end{theorem}



  From Theorem \ref{nth}, choosing the minimum from all the cases, we obtain the minimum number of closed factors in any factor of $f$ of length $F_n$, $n \geq 7$:

\begin{theorem}\label{finthy}
    For $n \geq 7$, $\min \{ \mathtt{Cl}(w) ~ | ~ w \in \text{Fac}(f)~ \& ~ |w|=F_n\} = \mathtt{Cl}(f_n)-2 F_{n-5}$. 
\end{theorem}

Now, to get the minimum number of closed factors in any factor of $f$, we need the following result.

\begin{proposition}\label{new1212}
 For $n \geq 7$,   let $v$ be a factor of $f$ such that $|v| = F_n + k$, where  $0 \leq k \leq F_{n-1}-1$. Let $s$ be the largest repeated suffix of $v$ and $r$ be the largest repeated suffix of $s$. Then,
    \begin{equation*}
|s|-|r| \geq
    \begin{cases}
        F_{n-4}, & \text{if } 0 \leq k \leq F_{n-3}-1\\
         F_{n-3}, & \text{if } F_{n-3} \leq k \leq F_{n-1}-1.
    \end{cases}
\end{equation*}
\end{proposition}

\begin{proof}
    We prove this by induction on $k$.
    \begin{itemize}
        \item Consider $0 \leq k \leq F_{n-3}-1$ : The statement is true for $k = 0$ by Proposition \ref{prop1}. Let the statement be true for factors of length $F_n + l$, where $l \leq k-1$.  Let $v$ be a factor of $f$ such that $|v| = F_n + k$, where $k \geq 1$.
    For some $x \in \Sigma$ and $ u \in \Sigma^+$, consider $v=x u$. Let $s_1$ be the largest repeated suffix of $u$ and $r_1$ be the largest repeated suffix of $s_1$. Then as $F_n \leq |u| < F_n + k$, by induction we have, $|s_1|-|r_1| \geq F_{n-4}$.
     Let $s$ be the largest repeated suffix of $v$ and $r$ be the largest repeated suffix of $s$. 
    Since $u$ is a suffix of $v$, $|s| \geq |s_1|$. Then, we have following cases:
    \begin{itemize}
        \item  If $|s|=|s_1|$, then as $u$ is a suffix of $v$, $s=s_1$. This implies $r=r_1$. Thus, $|s|-|r| \geq F_{n-4}$.

        \item   Let $|s|>|s_1|$. Then  as $s_1$ is the largest repeated suffix of $u$ and $v=xu$, $s$ must be the prefix of $v$.
    If possible, let $|s|-|s_1| \geq 2$. Then $x^{-1} s$ is the largest repeated suffix of $u$ which is a contradiction as $|x^{-1}s|-|s_1| \geq 1$. Thus, $|s|-|s_1| =1$. Since $s_1$ is a suffix of $s$, $|r| \geq |r_1|$. If $|r|= |r_1|$, then $|s|-|r| = |s_1|+1 - |r_1| > F_{n-4}$.
    Let $|r|>|r_1|$.  Then  as $r_1$ is the largest repeated suffix of $s_1$, $|s_1|_{r}=1$. This implies $r$ is a prefix of  $s$, i.e.,  $r$ is a prefix of  $v$. If possible, let $|r|-|r_1| \geq 2$. Then $x^{-1} r$ is the largest repeated suffix of $s_1$ which is a contradiction as $|x^{-1}r|-|r_1| \geq 1$. Thus, $|r|-|r_1| =1$.
    Therefore, $|s|-|r| = |s_1|+1 - |r_1|-1 = |s_1|-|r_1| \geq F_{n-4}$.
    \end{itemize}

        \item Consider $F_{n-3} \leq k \leq F_{n-1}-1$. 
        
        First we prove the statement for $k = F_{n-3}$. 


    For $n \geq 7$, let $v$ be a factor of $f$ of length $ F_n + F_{n-3}$. Let  $s$ be the largest repeated suffix of $v$ and $r$ be the largest repeated suffix of $s$.
  First, we describe the structure of $F_n + F_{n-3} + 1$ distinct factors of $f$ of length  $F_n + F_{n-3}$.
Now, 
    $f_n f_{n-1} f_n f_{n-1}= f_n f_n g_{n-1} f_{n-1} = f_n f_{n-3} f_{n-4} f_{n-3} f_{n-1} f_{n-2} f_{n-1}$ is a prefix of $f$. Then as $f_n f_n^{-}$ contains all distinct conjugates of $f_n$, all words in the form $u^{-1} f_n f_{n-3} u' $ are distinct factors of $f$ of length $F_n + F_{n-3}$, where $0 \leq |u|=|u'| \leq F_n-1$ and $u' \in Pref(f_{n-4} f_{n-3} f_{n-1} f_{n-2} f_{n-1})$. This gives us $F_n$ distinct factors of $f $ of length $F_n + F_{n-3}$.

    We now show that we can represent the remaining $F_{n-3}+1$ distinct factors of $f$ of length $F_n + F_{n-3}$ in the form $a_1 f_{n-3} f_{n-2} g_{n-1} a_1^{-1}$ and $q_1 f_{n-2} g_{n-1} q_2$,
    where $a_1 \in \Sigma \cap Suff(f_{n-2})$, $q_1 \in Suff(f_{n-3})$, $q_2 \in Pref(f_{n-1})$, $1 \leq |q_1| \leq F_{n-3}$, $0 \leq |q_2| \leq F_{n-3}-1$ and $|q_1|+|q_2| = F_{n-3}$.
    
    We have $f_n f_n g_{n-1} a_1^{-1} =f_n f_{n-2} f_{n-3} f_{n-2} g_{n-1} a_1^{-1} = f_n f_{n-1} f_n a_1^{-1} $ and $a_1 f_{n-3} f_{n-2} g_{n-1} a_1^{-1}=a_1 f_{n-3} f_n a_1^{-1}$. Then, by Assertions (\ref{impt2}), (\ref{impt3}) (\ref{impt6}), (\ref{impt7})   of Lemma \ref{lem11}, we have
    \begin{equation}\label{cxuh1}
        |f_n f_{n-1} f_n a_1^{-1}|_{a_1 f_{n-3} f_n a_1^{-1}}=1.
    \end{equation}
    We have $f_n f_n g_{n-1} q_2=f_n f_{n-2} f_{n-3} f_{n-2} g_{n-1} q_2 = f_n f_{n-1} f_n q_2 $ and $q_1 f_{n-2} g_{n-1} q_2=q_1 f_n q_2$.
    Then by Assertions (\ref{impt2}), (\ref{impt3}), (\ref{impt4})  of Lemma \ref{lem11}, we have 
    \begin{equation}\label{cxuh2}
        |f_n f_{n-1} f_n q_2|_{q_1 f_n q_2}=1  \text{ for each pair } q_1, q_2.
    \end{equation}
   
    Thus, (\ref{cxuh1}) and (\ref{cxuh2}) give us the desired representation of the remaining $F_{n-3}+1$ factors of length $F_n + F_{n-3}$.

    Now, for each of these distinct factors of $f$ of length $F_n + F_{n-3}$, we  calculate $|s|-|r|$ in the following:
    \begin{enumerate}
        \item Consider $|u|=0$ : Then, $v=u^{-1} f_n f_{n-3} u'=f_n f_{n-3} = f_{n-1} f_{n-1}$. As $f_{n-1}$ does not occur internally in $f_{n-1}f_{n-1}$, $s=f_{n-1}$. Then {by Lemma \ref{newlemcombine}}, $r=f_{n-3}$. So, $|s|-|r| = F_{n-2}$.

        \item Consider $1 \leq |u| \leq F_{n-4}$ : Then, $v=u^{-1} f_n f_{n-3} u'=  y f_{n-5} f_{n-4} f_{n-4} f_{n-5} f_{n-4} f_{n-5} f_{n-4} f_{n-4} f_{n-5} u$, where $f_{n-4}= uy$, where $0 \leq |y| \leq F_{n-4}-1$. Then, by Assertions (\ref{impt2}), (\ref{impt3}), (\ref{impt4})  of Lemma \ref{lem11}, $s=y f_{n-5} f_{n-4} f_{n-4} f_{n-5} u $. Now, we have the following cases:
        \begin{itemize}
            \item If $1 \leq |u| \leq F_{n-4}-2 $, then by Assertions (\ref{impt2}), (\ref{impt3}), (\ref{impt4})  of Lemma \ref{lem11}, $r=f_{n-4} f_{n-5} u$. Thus, $|s|-|r|= F_{n-3}+ |y|$, where $2 \leq |y| \leq F_{n-4}-1$.
            
            \item If $|u| = F_{n-4}-1$, then $u=f_{n-4} x_1^{-1}$, where $x_1 \in \Sigma$. This implies $s=x_1 f_{n-5} f_{n-4} f_{n-4} f_{n-5} f_{n-4} x_1^{-1} $. Then, by Assertions  (\ref{impt2}), (\ref{impt6}), (\ref{impt7}) of Lemma \ref{lem11}, $r=x_1 f_{n-5} f_{n-4} x_1^{-1}$. Thus, $|s|-|r|=F_{n-2}$.
            
            \item If $|u| = F_{n-4}$, then similar to the case $|u| = F_{n-4}-1$, $r= f_{n-5} f_{n-4} $. Thus, $|s|-|r|=F_{n-2}$. 
            
        \end{itemize}

        \item Consider $F_{n-4}+1 \leq |u| \leq F_{n-4} + F_{n-3}-2$ : Then, $v=u^{-1} f_n f_{n-3} u'= y f_{n-3} f_{n-3} f_{n-4} f_{n-3} f_{n-4} u_1 $, where $u'=f_{n-4} u_1$, $y \in Suff(g_{n-3})$, $2 \leq |y| \leq F_{n-3}-1$ and $u_1 \in Pref(f_{n-3})$. Then,  by Assertions  (\ref{impt2}), (\ref{impt3}), (\ref{impt4})  of Lemma \ref{lem11}, $s=y f_{n-3} f_{n-4} u_1$. Now, we have the following cases:
        
            \begin{itemize}

                \item If $2 \leq |y| \leq F_{n-4} $, then as $s=y f_{n-4} f_{n-5} f_{n-4} u_1$,  Assertions (\ref{impt2}), (\ref{impt3}), (\ref{impt4})   of Lemma \ref{lem11} gives $r=f_{n-4} u_1$. $|s|-|r|= |y| + F_{n-3}$.
            
                \item If $F_{n-4}+1 \leq |y| \leq F_{n-3}-1 $, then as $g_{n-3}=f_{n-5} f_{n-4}$, $y=t_1 f_{n-4}$, where $t_1 \in \Sigma^+ \cap Suff(f_{n-5})$. Then, $s=t_1 f_{n-4} f_{n-4} f_{n-5} f_{n-4} u_1$. 
                Then by Assertions  (\ref{impt2}), (\ref{impt3}), (\ref{impt4})  of Lemma \ref{lem11}, $r=t_1 f_{n-4} u_1$. Thus,  $|s|-|r|=F_{n-2}$. 
                
            \end{itemize}

        \item Consider $F_{n-4} + F_{n-3}-1  \leq |u| \leq  F_{n-4} + F_{n-3}$ : Let $|u| = F_{n-4} + F_{n-3}-1  $. Then, $v=u^{-1} f_n f_{n-3} u'= x_1 f_{n-3} f_{n-3} f_{n-4} f_{n-3} f_{n-4} f_{n-3} x_2^{-1}$, where $x_1, x_2 \in \Sigma$. Then, by Assertions   (\ref{impt2}), (\ref{impt3}), (\ref{impt4}) of Lemma \ref{lem11}, $s=f_{n-3} f_{n-4} f_{n-3} x_2^{-1}= f_{n-1} x_2^{-1} $. Then, by  Lemma \ref{newlemcombine}, $r=f_{n-3} x_2^{-1} $. Thus, $|s|-|r|=F_{n-2}$.\\
        Similarly, when $|u| = F_{n-4} + F_{n-3}$, then $s= f_{n-1}$, $r=f_{n-3} $ and $|s|-|r|=F_{n-2}$.

        \item Consider $F_{n-2}+1 \leq |u| \leq F_{n-1}-1$ : Then, $v=u^{-1} f_n f_{n-3} u'= t f_{n-2} f_{n-3} f_{n-4} f_{n-3} u_2 $, where $t \in \Sigma^+ \cap Suff(f_{n-3})$, $u_2 \in Pref(f_{n-3})$. Then, by Assertions   (\ref{impt2}), (\ref{impt3}), (\ref{impt4})   of Lemma \ref{lem11}, $s= f_{n-3} f_{n-4} f_{n-3} u_2$ and $r=f_{n-3} u_2$. Thus, $|s|-|r|=F_{n-2}$.

        \item Consider $|u| = F_{n-1}$ : Then, $v=u^{-1} f_n f_{n-3} u'=  f_{n-3}  f_{n-4}  f_{n-3}  f_{n-4}  f_{n-3}  f_{n-3}$. Then, by Assertions  (\ref{impt2}), (\ref{impt3}), (\ref{impt4}) of Lemma \ref{lem11}, $s= f_{n-3} f_{n-4} f_{n-3} f_{n-3}$ and $r=f_{n-3} f_{n-3}$. Thus, $|s|-|r|=F_{n-2}$. 

        \item Consider $F_{n-1} +1 \leq |u| \leq F_{n-1} + F_{n-4} -2$ : Then,  $v=u^{-1} f_n f_{n-3} u'\\
        = \delta  f_{n-5}  f_{n-4}  f_{n-4}  f_{n-5}  f_{n-4}  f_{n-4} f_{n-5}  f_{n-4} f_{n-5}  \gamma$, where $f_{n-4}  = \gamma \delta$, $1 \leq |\gamma| \leq F_{n-4}-2$. Then, by Assertions   (\ref{impt2}), (\ref{impt3}), (\ref{impt4})  of Lemma \ref{lem11}, $s= \delta  f_{n-5}  f_{n-4}  f_{n-4} f_{n-5}  f_{n-4} f_{n-5}  \gamma$ and $r=\delta f_{n-5}  f_{n-4} f_{n-5}  \gamma$. Thus, $|s|-|r|=F_{n-2}$.

        \item Consider $F_{n-1} + F_{n-4} -1 \leq |u| \leq F_{n-1} + F_{n-4} $ : Let $|u| = F_{n-1} + F_{n-4} -1$. Then,  $v=u^{-1} f_n f_{n-3} u'= x_3  f_{n-5}  f_{n-4}  f_{n-4}  f_{n-5}  f_{n-4}  f_{n-4} f_{n-5}  f_{n-4} f_{n-5}  f_{n-4} x_2^{-1}$, where $x_2 \in \Sigma \cap Suff(f_{n-4})$. Then,   by Assertions (\ref{impt2}), (\ref{impt3}), (\ref{impt4}) of Lemma \ref{lem11}, $s= f_{n-4} f_{n-5}  f_{n-4} x_2^{-1} $ and $r= f_{n-4}  x_2^{-1}$. Thus, $|s|-|r|=F_{n-3}$. Similarly, when $|u| = F_{n-1} + F_{n-4}$, then $s= f_{n-4} f_{n-5}  f_{n-4} $ and $r= f_{n-4} $. Thus, $|s|-|r|=F_{n-3}$.

        \item Consider  $F_{n-1} + F_{n-4} +1 \leq |u| \leq F_n -2$ : Then,  $v=u^{-1} f_n f_{n-3} u'= \delta_1  f_{n-3}  f_{n-4}  f_{n-3}  f_{n-3}  f_{n-4}   \gamma_1$, where $g_{n-3}  = \gamma_1 \delta_1$, $1 \leq |\gamma_1| \leq F_{n-3}-2$.  Then,  by Assertions (\ref{impt2}), (\ref{impt3}), (\ref{impt4})  of Lemma \ref{lem11}, $s=  f_{n-3} f_{n-4}  \gamma_1$ and $r=  f_{n-4}  \gamma_1$. Thus, $|s|-|r|=F_{n-3}$.

        \item Consider $|u|=F_n -1$ : Then, $v=u^{-1} f_n f_{n-3} u'= x_3 f_{n-1} f_{n-1} x_3^{-1}$, where $x_3 \in \Sigma$. Then, $s=  f_{n-1} x_3^{-1}$ and $r=  f_{n-3}  x_3^{-1}$. Thus, $|s|-|r|=F_{n-2}$.

        \item Consider $v = a_1 f_{n-3} f_{n-2} g_{n-1} a_1^{-1}$, where $a_1 \in \Sigma \cap Suff(f_{n-2})$ : Then,  by Assertions (\ref{impt2}), (\ref{impt3}), (\ref{impt4})   of Lemma \ref{lem11}, $s=  a_1 f_{n-3} f_{n-2}  a_1^{-1}$ and $r=  f_{n-2}  a_1^{-1}$. Thus, $|s|-|r|=F_{n-3}+1$.

        \item Consider $v = q_1 f_{n-2} g_{n-1} q_2$, where $q_1 \in Suff(f_{n-3})$, $q_2 \in Pref(f_{n-1})$, $1 \leq |q_1| \leq F_{n-3}$, $0 \leq |q_2| \leq F_{n-3}-1$ and $|q_1|+|q_2| = F_{n-3}$ : Then, $v = q_1 f_{n-2} g_{n-1} q_2 = q_1 f_{n-2} f_{n-3}   f_{n-2} q_2$. Then,  by Assertions (\ref{impt2}), (\ref{impt3}), (\ref{impt4})   of Lemma \ref{lem11}, $s=  q_1 f_{n-2} q_2$. We have the following cases:
        \begin{itemize}
            \item If $0 \leq |q_2| \leq F_{n-3}-2$:  Then, by Assertions (\ref{impt2}), (\ref{impt3}), (\ref{impt4})  of Lemma \ref{lem11}, $r=q_1 f_{n-4} q_2$. Thus,  $|s|-|r|=F_{n-3}$.
            \item If $ |q_2| = F_{n-3}-1$:  Then,  by Assertions (\ref{impt2}), (\ref{impt3}), (\ref{impt4}) of Lemma \ref{lem11}, $r=f_{n-3} a_2^{-1}$, where $a_2 \in \sigma$. Thus,  $|s|-|r|=F_{n-2} +1$.
        \end{itemize}   
    \end{enumerate}

      Similar to the inductive steps of the case $0 \leq k \leq F_{n-3}-1$, we now prove that for $k>F_{n-3}$,  $|s|-|r| \geq F_{n-3}$.

\end{itemize}
\end{proof}

From Theorem \ref{finthy}, we know that  $$\min \{ \mathtt{Cl}(w) ~ | ~ w \in \text{Fac}(f)~ \& ~ |w|=F_n\} = \mathtt{Cl}(f_n)-2 F_{n-5}$$ for $n \geq 7.$ Then,
from Lemma \ref{lem0} and Proposition \ref{new1212}, it is clear that if  $v$ is a factor of $f$ of length $F_n+k$, where  $n \geq 7$, $0 \leq k \leq F_{n-3}-1$, then $\mathtt{Cl}(v) \geq \mathtt{Cl}(f_n) - 2 F_{n-5} + k F_{n-4}$. 
Similarly, from Lemma \ref{lem0} and Proposition \ref{new1212}, we have, if  $v$ is a factor of $f$ of length $F_n+k$, where $F_{n-3} \leq k \leq F_{n-1}-1$, then $\mathtt{Cl}(v) \geq
\mathtt{Cl}(f_n) - 2 F_{n-5} + (F_{n-3}-1) F_{n-4} + (k-F_{n-3}+1) F_{n-3}$. Thus, we have the following:

\begin{theorem}\label{Finalth}
    Let $v$ be a factor of $f$ of length $F_n+k$, where $0 \leq k \leq F_{n-1}-1$ and $n \geq 7$. Then,
    \begin{equation*}
\mathtt{Cl}(v) \geq
    \begin{cases}
        \mathtt{Cl}(f_n) - 2 F_{n-5} + k F_{n-4}, & \text{if } 0 \leq k \leq F_{n-3}-1\\
        \mathtt{Cl}(f_n) - 2 F_{n-5} + (F_{n-3}-1) F_{n-4} + (k-F_{n-3}+1) F_{n-3}, & \text{if } F_{n-3} \leq k \leq F_{n-1}-1.
    \end{cases}
\end{equation*}
\end{theorem}

\subsection{Lower and Upper bound of closed-rich constant of $f$}\label{sec6}

In \cite{parshina2024finite}, it was proved that $f$ is an infinite closed-rich word. We provide the following lower and upper bounds for its closed-rich constant:
\begin{theorem}\label{fibthfg}
     The following bounds hold for the closed rich constant $C_f$ of the Fibonacci word $f$:
     \[  \frac{\phi^3+3}{\phi^9} < C_f  \leq  \frac{5\phi+3}{\phi^2(\phi^3+2)^2}, \]
    where $\phi = \frac{1+\sqrt{5}}{2}$.
\end{theorem}

\begin{remark} The bounds from the above theorem are quite close: $ 0.09519 \leq C_f\leq 0.10893 $. 
\end{remark}

\begin{proof}
     We first prove the upper bound of $C_f$.  For this, we will use numerical experiments. Then we will discuss the lower bounds for $C_f$ using Remark \ref{remsup} and Theorem \ref{Finalth}.
    \begin{itemize}
        \item Upper bound for $C_f : $

         To get an upper bound of $C_f$, we now compute the $\mathtt{Cl}(f_{n-3}f_n f_{n-3}^{--})$ for $n \geq 7$.
        First note that $f_{n-3} f_n f_{n-3}^{--} \in Fac(f)$. Indeed, $f_{n-3}$ is a suffix of $f_{n+2}$ and $f_n f_{n-3}^{--}$ is a prefix of $ f_{n+1} $; so, their concatenation is a factor of $f_{n+2}f_{n+1}$. As $f_n f_{n-3}^{--} \in Pref(f)$, Theorem \ref{Prethecl} gives $\mathtt{Cl}(f_n f_{n-3}^{--})$. We now compute $\mathtt{Cl}(f_{n-3} f_n f_{n-3}^{--})$, extending the factor $f_n f_{n-3}^{--}$ to the left letter by letter and applying Lemma \ref{lem01}. Let $xf_n f_{n-3}^{--} \in Suff(f_{n-3} f_n f_{n-3}^{--})$, where $x \in Suff(f_{n-3})$. 
        Since $xf_n f_{n-3}^{--} = x f_{n-1} f_{n-1}^{--}$, the word $x f_{n-1}^{--}$ is a repeated prefix of $xf_n f_{n-3}^{--}$. By Lemma \ref{lem11}(\ref{impt6}), the word $x f_{n-1}^{--}$ is the largest repeated prefix of $xf_n f_{n-3}^{--}$. 
        Also, as $x f_{n-1}^{--} = x f_{n-2} f_{n-3}^{--}$,  $x f_{n-2}^{--}$ is a repeated prefix of $x f_{n-1}^{--}$. By Lemma \ref{lem11}(\ref{impt6}), $x f_{n-2}^{--}$ is the largest repeated prefix of $x f_{n-1}^{--}$. Then, $$\mathtt{Cl}(x f_n f_{n-3}^{--}) - \mathtt{Cl}( ^{-}x f_n f_{n-3}^{--}) = |x f_{n-1}^{--}| - |x f_{n-2}^{--}| = F_{n-3}.$$ Since $x$ is an arbitrary suffix of $f_{n-3}$, applying this iteratively to suffices of length 1, 2, \ldots, $F_{n-3}$, we obtain $$\mathtt{Cl}(f_{n-3} f_n f_{n-3}^{--}) = \mathtt{Cl}(f_n f_{n-3}^{--}) + F_{n-3} F_{n-3} = 1+F_{n-3} F_n.$$ Therefore, $$\displaystyle C_f \leq \lim_{n \rightarrow \infty}\frac{1+F_{n-3} F_n}{(F_n + 2F_{n-3}-2)^2} =\frac{5\phi+3}{\phi^2(\phi^3+2)^2} \approx 0.10893. $$

        \item Lower bound of $C_f$ : 
        
        The main strategy of the proof is as follows. Let $v \in Fac(f)$ and $p \in Pref(f)$ such that $|v|=|p|=F_n+k$, where $0 \leq k < F_{n-1}$ and $n \geq 7$. From Theorem \ref{Finalth},  we have
         \begin{equation*}
            \frac{\mathtt{Cl}(v)}{|v|^2} \geq
            \begin{cases}
             \frac{1 + F_{n-3} F_{n-2} + F_{n-4} ( F_{n-3} + k +2)}{(F_n+k)^2}, & \text{if } 0 \leq k \leq F_{n-3}-1\\
            \frac{1 + F_{n-3} ( 3F_{n-4} +k)  + F_{n-2}}{(F_n+k)^2}, & \text{if } F_{n-3} \leq k \leq F_{n-1}-1.
            \end{cases}
        \end{equation*}

        We first prove that  $ \frac{1 + F_{n-3} F_{n-2} + F_{n-4} F_{n-3} + k F_{n-4} +2 F_{n-4}}{(F_n+k)^2} >  \frac{\phi^{-5} + l \phi^{-4} + \phi^{-7}}{ (l+1)^2}$ when $0 \leq k \leq F_{n-3}-1$ and $0 \leq l < \frac{1}{\phi^3}$.
        Next we show that 
         $\frac{1 + F_{n-3} ( 3F_{n-4} +k)  + F_{n-2}}{(F_n+k)^2} > \frac{ 3 \phi^{-7} + l \phi^{-3} }{ (l+1)^2}$ when $F_{n-3} \leq k \leq F_{n-1}-1$ and $\frac{1}{\phi^3} \leq l < \frac{1}{\phi}$. 
         Then, we prove that $\inf \{ \frac{\phi^{-5} + l \phi^{-4} + \phi^{-7}}{ (l+1)^2}, \frac{ 3 \phi^{-7} + l' \phi^{-3} }{ (l'+1)^2}   : 0 \leq l < \frac{1}{\phi^3}, \frac{1}{\phi^3} \leq l' < \frac{1}{\phi}  \} = \frac{\phi^3+3}{\phi^9}$.
         This implies that for any $v \in Fac(f)$ of length $F_n+k$ with $n \geq 7$ and $0 \leq k < F_{n-1}$,   $\frac{\mathtt{Cl}(v)}{|v|^2} > \frac{\phi^3+3}{\phi^9} $. Then, using the initial values obtained by computer and provided in Table \ref{tnb1}, we show that $\sup \left\{ C: \mathtt{Cl}(w)\geq C|w|^2 \mbox{ for each } w \in  Fac(f)  \right\} > \frac{ \phi^3+3 }{\phi^9}$, i.e.,   $C_f > \frac{ \phi^3+3 }{\phi^9}$ (by Remark \ref{remsup}).
         
        First we recall that $F_n = \frac{\phi^{n+2} - (1-\phi)^{n+2}}{\sqrt{5}}$ for $n \geq -1$.
         When $0 \leq k \leq F_{n-3}-1$, we set $l=k\frac{\sqrt{5}}{\phi^{n+2}}$; note that $0 \leq l < \frac{1}{\phi^3}$.
        Using $F_n = \frac{\phi^{n+2} - (1-\phi)^{n+2}}{\sqrt{5}}$ and $\phi (1-\phi) = -1$, we have : 
        \begin{align*}
            \frac{1 + F_{n-3} F_{n-2} + F_{n-4} F_{n-3} + k F_{n-4} +2 F_{n-4}}{(F_n+k)^2} 
            =  \frac{x+ a' }{y+ b'},
        \end{align*}
  where 
    \begin{flalign*} & x= \phi^{2n-1} + l \phi^{2n} + \phi^{2n-3}, \\ 
  & y= \phi^{2n+4} (l+1)^2, \\
  & a'= 5 + (1-\phi)^{2n-1} + (1-\phi)^{2n-3} -l \phi^{4} (-1)^n + 2\sqrt{5} \phi^{n-2} - 2\sqrt{5} (1-\phi)^{n-2}, \\ & b'= (1-\phi)^{2n+4} -2(-1)^n -2l (-1)^n.  \end{flalign*} 
  Now, 
   $b'  < 3$
   and
    $a' >  2\sqrt{5} \phi^{n-2}$.
   In fact, $\frac{x}{y}$ gives the asymptotics of the lower bound, and the terms $a'$ and $b'$ are bounded. However, by the definition of a closed-rich constant we need to show the inequality $\frac{x+ a' }{y+ b'} > \frac{x}{y}$; we do it separately for positive and negative values of $b'$.
    

   Consider $b' >0$. Then $\frac{x+ a' }{y+ b'} > \frac{x}{y}$ iff $\frac{a'}{b'} > \frac{x}{y}$.
   Now, $\frac{x}{y} = \frac{\phi^{2n-1} + l \phi^{2n} + \phi^{2n-3}}{\phi^{2n+4} (l+1)^2} = \frac{\phi^{-5} + l \phi^{-4} + \phi^{-7}}{ (l+1)^2} < 1$, and
 $\frac{a'}{b'} > \frac{2\sqrt{5} \phi^{n-2}}{3} >1$. So clearly, $\frac{a'}{b'} > \frac{x}{y}$.

   Consider $b'\leq 0$. We have $a'y -b'x >0$, i.e., $xy + a'y >xy +b'x$, i.e., $(x+a')y > (y+b')x$, i.e., $\frac{x+a'}{y+b'} > \frac{x}{y}$ (since $y>-b'$). This implies  $\frac{x+a'}{y+b'} > \frac{x}{y}$.

   Thus, for $0 \leq l < \frac{1}{\phi^3}$, $ \frac{1 + F_{n-3} F_{n-2} + F_{n-4} F_{n-3} + k F_{n-4} +2 F_{n-4}}{(F_n+k)^2} > \frac{x}{y} =  \frac{\phi^{-5} + l \phi^{-4} + \phi^{-7}}{ (l+1)^2}$.\\
   Similarly, we show that for $\frac{1}{\phi^3} \leq l < \frac{1}{\phi}$, $\frac{1 + F_{n-3} ( 3F_{n-4} +k)  + F_{n-2}}{(F_n+k)^2} > \frac{ 3 \phi^{-7} + l \phi^{-3} }{ (l+1)^2}$.

   Consider \begin{equation*}
        h(l)=
        \begin{cases}
        \frac{\phi^{-3}(\phi^{-2}+\phi^{-4})+ l \phi^{-4}}{(1+l)^2}, & \text{if } 0 \leq  l < \frac{1}{\phi^3}\\
         \frac{3\phi^{-7} + l \phi^{-3} }{(1+l)^2}, & \text{if } \frac{1}{\phi^3} \leq  l \leq \frac{1}{\phi}.
        \end{cases}
    \end{equation*}
    By a straightforward analysis one can see that $h(l)$ attains its minimum at  $l=\frac{1}{\phi}$, and $h(\frac{1}{\phi}) = \frac{ \phi^3+3 }{\phi^9}$. Then using the values from Table \ref{tnb1} for $n<7$, we get $\frac{ \phi^3+3 }{\phi^9}$ as a lower bound:
$$ \sup \left\{ C: \mathtt{Cl}(w)\geq C|w|^2 \mbox{ for each } w \in  Fac(f)  \right\} > \frac{ \phi^3+3 }{\phi^9}.$$
Thus, by Remark \ref{remsup}, $C_f > \frac{ \phi^3+3 }{\phi^9}$.

     \end{itemize}

 \end{proof}

 



\begin{table}[H]
    \centering
    \begin{minipage}{.3\textwidth}
      \centering
      \begin{tabular}{|p{.7cm}|p{.83cm}|p{.5cm}|p{.6cm}|}
        \hline
      $n$ & $\mathtt{PCl}(n)$ & $M_n$ & $\frac{M_n}{n^2}$ \\
        \hline
        1 & 2 & 2 & 2 \\
        2 & 3 & 3 & 0.75 \\
        3 & 4 & 4 & 0.44 \\
        4 & 5 & 5 & 0.31 \\
        5 & 7 & 6 & 0.24 \\
        6 & 9 & 9 & 0.25 \\
        7 & 11 & 10 & 0.2 \\
        8 & 13 & 13 & 0.2 \\
        9 & 16 & 15 & 0.19 \\
        10 & 19 & 17 & 0.17 \\
        11 & 22 & 22 & 0.18 \\
        \hline
      \end{tabular}
    \end{minipage}%
    \hspace{.01\textwidth} 
    \begin{minipage}{.3\textwidth}
      \centering
      \begin{tabular}{|p{.7cm}|p{.83cm}|p{.5cm}|p{.6cm}|}
       \hline
       $n$ & $\mathtt{PCl}(n)$ & $M_n$ & $\frac{M_n}{n^2}$ \\
        \hline
        12 & 25 & 24 & 0.17 \\
        13 & 28 & 26 & 0.15 \\
        14 & 31 & 31 & 0.16 \\
        15 & 36 & 34 & 0.15 \\
        16 & 41 & 37 & 0.14 \\
        17 & 46 & 40 & 0.14 \\
        18 & 51 & 48 & 0.15 \\
        19 & 56 & 56 & 0.16 \\
        20 & 61 & 59 & 0.15 \\
        21 & 66 & 62 & 0.14 \\
        22 & 71 & 65 & 0.13 \\
        \hline
      \end{tabular}
    \end{minipage}
     \hspace{.01\textwidth} 
    \begin{minipage}{.3\textwidth}
      \centering
      \begin{tabular}{|p{.7cm}|p{.83cm}|p{0.5cm}|p{.6cm}|}
        \hline
       $n$ & $\mathtt{PCl}(n)$ & $M_n$ & $\frac{M_n}{n^2}$ \\
        \hline
        23 & 76 & 73 & 0.14 \\
        24 & 81 & 81 & 0.14 \\
        25 & 89 & 86 & 0.14 \\
        26 & 97 & 91 & 0.13 \\
        27 & 105 & 96 & 0.13 \\
        28 & 113 & 101 &  0.13 \\
        29 & 121 & 106 & 0.13 \\
        30 & 129 & 119 & 0.13 \\
        31 & 137 & 132 & 0.14 \\
        32 & 145 & 145 & 0.14 \\
        33 & 153 & 150 & 0.14 \\
        \hline
      \end{tabular}
    \end{minipage}
    \caption{Values of $\mathtt{PCl}(n)$, $M_n = \min \{\mathtt{Cl}(w): w\in Fac(f) \cap \Sigma^n\}$ and $\frac{M_n}{n^2}$ for factors of length $n < F_7$.}
    \label{tnb1}
\end{table}

\section {Conclusions and open problems}\label{sec7}

Let  $C_{sup} =  \sup \{C_u: C_u \text{ is the closed-rich constant of an infinite closed-rich word } u \}$. 
From Theorems \ref{mthsd} and \ref{fibthfg}, we have the following bounds for $C_{sup}$:

\begin{corollary}
     $0.09519 <  C_{sup} \leq  0.165952$.
\end{corollary}

With computational experiments we observed that $f$ provides a better lower bound for $C_{sup}$ than several well-known infinite closed-rich words. For instance, for the Thue-Morse word (TM) we obtained $C_{TM} \leq 0.09 $, for ternary Thue-Morse word (TTM)  $C_{TTM} \leq 0.092$, for the period-doubling word (PD) $C_{PD}  \leq 0.093 $, for Dejean’s word (D)  $C_{D}  \leq 0.08 $, for  Mephisto-Waltz word (MW) $C_{MW}  \leq 0.08 $,  for Leech’s word (LE) $C_{LE}  \leq 0.07 $, for Fibonacci-Thue-Morse word (FTM) $C_{FTM}  \leq 0.062 $,   etc.



Based on numerical experiments, we propose the following conjecture 
regarding 
$M_n = \min \{ \mathtt{Cl}(w): w \in Fac(f) \cap \Sigma^n\}$. 
The following conjecture gives a formula for its first difference sequence $R_n = M_{n} - M_{n-1}$ (it is not hard to derive from it an explicit formula for $M_n$).

\begin{conjecture} \label{conj0}
\begin{align*}
    \{R_n\}_{n \geq 1}  & =1, 1, 1, 1, F_0, F_2, F_0, F_2, F_1, F_1, F_3, F_1, F_1, F_3, F_2, F_2, F_2, F_4, F_4, F_2, F_2,\\
    &F_2, F_4, F_4, \cdots  \\
    & = 1, 1, 1, 1, \displaystyle \prod_{m=0}^{\infty} {{F_m}^{F_m}, {F_{m+2}}^{F_{m-1}}, {F_m}^{F_m}, {F_{m+2}}^{F_{m-1}}  }
\end{align*}

    
\end{conjecture}






Based on Conjecture \ref{conj0}, we conjecture the value of $C_f$:
\begin{conjecture} \label{conj2}
    For $f$, the closed-rich constant $C_f= \frac{5\phi+3}{\phi^2(\phi^3+2)^2} \approx 0.10893$.
\end{conjecture}

\begin{remark} From the conjectured sequence, the minimum value of closed factors  occurs for the lengths $F_n + 2 F_{n-3} - 2$ (in fact, between these and the next lengths the first difference function increases). Computer experiments show that for each of these lengths there is only one factor giving the minimal value. Moreover, these factors are palindromes of a special kind: they have bispecial factors of length $F_n -2$ in the middle, which are continued in the unique way to the left and to the right with length $F_{n-3}$. These factors give the upper bound on $C_f$ from Theorem \ref{fibthfg}. However, for now we do not have a proof of the fact that these factors give the minimum.\end{remark}

The question about the closed-rich constant of the Fibonacci word can be naturally extended to closed-rich Sturmian words:

\begin{question}
    Given a Sturmian word $s$ of slope $\alpha$ such that its continued fraction expansion has bounded partial quotients. In \cite{parshina2024finite} is has been shown that $s$ is closed-rich. What is the closed-rich constant of $s$?
\end{question}


\begin{question}
    Is it true that any real number in $(0,C_{sup})$ is a closed-rich constant of some infinite closed-rich word? If not, is the set of closed-rich constants dense in $(0,C_{sup})$?
\end{question}

\section{Acknowledgement} 

The study was supported by the Russian Science Foundation, project $23-11-00133$.


\bibliographystyle{abbrv}
\bibliography{refer.bib}

\end{document}